\documentclass[11pt,oneside,a4paper]{article}

\usepackage[margin=1.3in]{geometry}
\usepackage{csquotes}

\usepackage{latexsym,upref,amsmath,amssymb,amsfonts,amsthm}
\usepackage{graphicx}

\usepackage[normalem]{ulem}

\usepackage[pdfusetitle]{hyperref}
\usepackage{cleveref}

\usepackage{caption}
\usepackage{subcaption}
\usepackage{float}
\usepackage{enumerate}
\usepackage{dsfont}

\usepackage{ifthen}

\newcommand\DOI[1]{{\tt DOI:\href{https://doi.org/#1}{#1}}}

\def\P{\mathcal{P}}
\def\N{\mathcal{N}}
\newcommand{\old}[1]{}

\theoremstyle{plain}
\newtheorem{theorem}{Theorem}[section]
\newtheorem{lemma}[theorem]{Lemma}
\newtheorem{corollary}[theorem]{Corollary}

\theoremstyle{definition}
\newtheorem{definition}[theorem]{Definition}
\newtheorem{observation}[theorem]{Observation}
\newtheorem{example}[theorem]{Example}

\usepackage{tikz}
\usetikzlibrary{calc,arrows,shapes}
\tikzset{every picture/.style=very thin}
\usepackage{ifthen}
\usepackage{authblk}

\title{\bfseries Degree-preserving graph dynamics - a versatile process to
construct random networks\footnotemark[0]\footnotetext{PLE and TRM were supported in part by the National Research, Development and Innovation Office --- NKFIH grant SNN 135643, K 132696. SRK and ZT were supported by the NSF grant IIS-1724297.}}
\author[1]{P\'eter L. Erd\H{o}s}
\author[2]{Shubha R. Kharel}
\author[1]{Tam\'as R. Mezei}
\author[2]{Zoltan~Toroczkai}

\affil[1]{\small Dept.\ of Combinatorics and applications, Alfréd Rényi Inst.\
    of Math.\ (LERN),\protect\\ Reáltanoda utca 13--15, H-1053 Budapest, Hungary.\protect\\\texttt{<erdos.peter,
mezei.tamas.robert>@renyi.hu}}
\affil[2]{\small Department of Physics, 225 Nieuwland Science Hall, Notre Dame,
    IN 46556, USA.\protect\\\texttt{<skharel,toro>@nd.edu}}

\begin{document}

\maketitle

\begin{abstract}
    Real-world networks evolve over time via additions or removals of vertices and
    edges. In current network evolution models, vertex degree varies or grows
    arbitrarily. A recently introduced \emph{degree-preserving network growth}
    (DPG) family of models preserves vertex degree, resulting in structures
    significantly different from and more diverse than previous models
    ([\emph{Nature Physics 2021, \DOI{10.1038/s41567-021-01417-7}}]).
    Despite its degree preserving property, the
    DPG model is able to replicate the output of several well-known
    real-world network growth models. Simulations showed that many well-studied real-world
    networks can be constructed from small seed graphs.
    
    \medskip

    Here we start the development of a rigorous mathematical theory underlying the DPG family of network
    growth models.
    We prove that the degree sequence of the output of some of the well-known,
    real-world network growth models can be reconstructed via the DPG process,
    using proper parametrization. We also show that the general problem of
    deciding whether a simple graph can be obtained via the DPG process from a
    small seed (DPG feasibility) is, as expected, \textsc{NP}-complete. It is an
    important open problem to uncover whether there is a
    structural reason behind the DPG-constructibility of real-world networks.
    
    \smallskip
    
\end{abstract}

\section{Introduction}\label{sec:intro}

Many network models have been introduced in the literature, from the
configuration model of Bollobás~\cite{Bollobas} and Molloy and Reed~\cite{Molloy} through the Watts-Strogatz
small-world networks~\cite{WS98}, the Chung-Lu models~\cite{ACL00}, to the
\textsc{IncPower} model of Arman, Gao, and Wormald~\cite{AGW19} and, 
arguably the most popular model by Barab\'asi and Albert~\cite{BA99}, also 
called the preferential attachment (PA) model. In most of the growth models 
(including in the PA), the incoming vertex forms connections with a 
select number of vertices from the existing network, therefore also 
increasing the degrees of those vertices. Accordingly, this makes the degree of
a vertex in the network dependent on the current size (number of vertices) of the {\em whole
network}, although with increases happening with smaller and smaller probability as the network grows. While this is not an issue for networks in which the creation and maintenance of edges does not bear a cost to the vertices (such as citation networks), it becomes unrealistic for physical networks. This is because in such networks link formation and maintenance bears a cost to the vertices (usually a local cost), leading to degree saturation, due to natural budget limitations in physical systems. Moreover, in models like the PA, vertex degree is not an intrinsic property of the vertex, but it is imposed externally and globally, which, again, is often unrealistic. 
\medskip

Here we describe a family of novel network growth models which considers vertex degree an \emph{intrinsic property} of the vertex and the process of network growth ``respects'' the degrees of the vertices already incorporated in the network. There is clearly a large multitude of model classes with this desired property, for example, models that only consider a degree limit/capacity (a saturation value) to be an intrinsic property of the vertex (otherwise the degrees can vary up to that limit), or models, in which the degrees are fixed and are an intrinsic property of the vertices, staying constant throughout the growth process (once the vertices fully joined the network). Although both types of classes share the same concept, here we focus on the latter, due to their simplicity and mathematical tractability. 
This type of model class was recently introduced under the name of  \emph{degree-preserving network growth} (DPG) in~\cite{DPG21}. In the DPG model family the new (or ``incoming'') vertices join the network with a preset degree, called here \emph{proper degree}, or \emph{p-degree} in short, by connecting to the vertices of the existing network (the ``old'' vertices), in a way that their degrees stay unchanged and the graph stays a simple graph (described below).

\medskip

Network growth models in which degrees are an intrinsic property of the vertices, are useful and needed from a modeling perspective.
An example is the case of  chemical compounds: here, if a vertex represents
an atom, then necessarily, its degree, which is the atom's valency (i.e., number of
chemical bonds it can form) must stay fixed during the process; however, chemical complexes
can, in principle, be arbitrarily large. Another example is the class of networks in which
(some, or all of) the existing vertices cannot accept additional connections because their
connectivity is saturated, such as in social networks, infrastructure networks, 
or, as described above in any physical network where the formation and maintenance of connections 
bears a cost. 

\medskip

The DPG dynamics can be described in the simplest form for even degrees: let $G$ be a simple graph. In a step, a new vertex $w$ joins the graph by removing $k$ pairwise disjoint edges of $G$, i.e., a matching, followed by connecting $w$ to the end vertices of the $k$ removed edges. The degree of the newly inserted vertex is $2k$. This step does not join
two vertices that are non-adjacent, and furthermore, the degrees of vertices in $G$ are
not changed. This operation is called a \emph{degree preserving growth} step
(\textbf{DP-step} for short). The \emph{degree-preserving growth} process
iteratively repeats DP-steps, starting with an arbitrarily chosen graph; the
resulting process is what we generally refer to as \textbf{DPG} dynamics. 
In \Cref{sec:def} below we provide the general description corresponding to the inclusion of vertices with degrees of arbitrary parity.

\medskip

A specific case of this process is not completely new. If the degree of the 
inserted vertex is a {\em constant $2k$}, we get a dynamic model for (relatively) 
random $2k$-regular graphs. The case $k=2$ played an important role for client-server 
architectures in peer-to-peer networking called SWAN technology~\cite{Bou03}, 
using the so-called ``clothespinning'' procedure (a special case of the general 
DP-step) and its inverse. We will return back to this fact in \Cref{sec:reg}.

\medskip

The different possible strategies for dynamically choosing the degree $k$ of the
incoming vertex and the different modalities of joining them to the existing network 
provide a large collection of rather different growth models and very different 
kinds of networks (see~\cite{DPG21}).
Therefore it is natural to ask what kinds of networks can be constructed by the various
DPG models. To answer this question, it is necessary to fully describe
the \emph{inverse} operation of a DP-step. Pick a vertex $w$ in graph $G$ and
examine its neighborhood graph $\Gamma(w)$. Assume that there exists a
\emph{perfect matching} $M$ in the \emph{complement graph} $\bar \Gamma(w)$ of
$\Gamma(w)$. Let us delete $w$ --- together with its adjacent edges --- and add
$M$ to $G\setminus w$; call this move a \textbf{DP-removal} and the resulting
graph $G_w$. The original graph $G$ can be reproduced from $G_w$ with a DP-step
with $k:=|M|$. (The complete definition of the DP-removals will be described in
\Cref{sec:def}.)

\medskip

In~\cite{DPG21}, many real-world networks were studied, if whether they can be
constructed from small initial ``kernel'' graphs via the DPG-process. 
Note that there can be many DP paths leading to the same target graph, 
starting from different small kernel graphs or possibly even from the same small kernel graph.
The very surprising computational finding was that the overwhelming majority of those real-life networks have this property. 
This numerical observation directly leads to the impression that finding a long sequence
of DP-removals for a given network should not be hard. One of the results of
this paper (\Cref{sec:NP}) is that this \textbf{inverse DPG problem} is actually
\textsc{NP}-complete.  We believe that this apparent contradiction means that 
all those real-life networks may share a yet unknown
common structural property that makes them DPG feasible. This indicates that the DPG
dynamics captures something about these networks that is not explained by other
models.

\medskip

In this paper we will also discuss new results on different DPG models
(see~\Cref{sec:kind}) as well as some extensions of the algorithmic and
stochastic considerations of DPG processes defined in~\cite{DPG21}.

\section{Definitions}\label{sec:def}

Let $G$ be an \emph{simple} graph, i.e., 
without loops and parallel edges. We will construct a sequence 
${(G_i)}_{i=0}^\infty$ of networks via a DPG process, where $G_0=G$ and 
$V(G_i)\subset V(G_{i+1})$ for any $i\ge 0$, and $G_{i+1}$ has exactly one additional vertex compared to $G_i$. 
In \Cref{sec:intro} we described how to add a new vertex of even degree, which we
now generalize to allow inserting odd p-degree vertices as well. Such an addition, however, cannot be 
achieved in one step: deleting a matching and connecting the new vertex $w$ 
to the end-points of the edges from the matching, endows $w$ with an even degree. 
One possible, and arguably the simplest solution is that when we want to add 
a new vertex $w$ of degree $2k+1$, we connect  $2k$ edges of $w$ to the graph
(as described above), and introduce an imaginary edge, called the \emph{stub-edge}, which maintains the missing degree of $w$.  When another new odd-degree vertex $w'$ arrives at a later step, the algorithm always connects it to the vertex associated with the existing stub-edge, forming a new edge in the new graph. 
\medskip

In order to make clearer the
description of the process of adding degrees of arbitrary parity, we make use of the notion of p-degree, introduced in the previous section. Accordingly, the p-degree is nothing but the intrinsic degree of the incoming vertex. We distinguish the p-degree from the vertex's actual degree in the network, the latter corresponding to the number of edges incident on the vertex, which is one less than the p-degree if there is a stub-edge incident on our vertex; otherwise the two are identical. The vertex that has a stub-edge will be called ``degree-deficient'' vertex. 

\medskip

To help with the bookkeeping of the degree-deficient vertex, extend the current
network $G_i$ to $G^s_i$, which contains an extra vertex $s$ of degree $0$ or $1$,
called the \emph{stub-node}. If $G_i$ contains a degree-deficient vertex $x$, then $G_i^s$
contains the edge $xs$; otherwise $s$ is isolated from $G_i$. The vertex $s$ and 
the \emph{stub-edge} $xs$ do not belong to $G_i$. Note, all vertices (except $s$) in $G_i^s$ are connected 
according to their p-degrees. We introduce the 
\emph{lifting operation} to provide a unified description of different types of 
DP-steps.  The current network is $G_i$, and we want to construct a network 
$G_{i+1}$ which will contain the incoming vertex $w$.

\medskip
To begin with, we fix some \emph{principles} we want to uphold during the DPG process:
\begin{enumerate}[\bfseries {Principle} 1)]
    \item For any $w\in V(G_i^s)\setminus \{s\}$ the degree of $w$ is constant: $d_{G_i^s}(w)=d_{G_{i+1}^s}(w)$.
    \item In each DP-step, the edges of a matching will be removed; specifically, the difference
        $G_i\setminus G_{i+1}[V(G_i)]$ is a matching for any $i$.
    \item After any completed DP-step, there will be at most 1 stub-edge in the produced network.
    \item If $u,v\in V(G_i)$ and $uv\notin E(G_i)$, then $uv\notin E(G_{i+1})$.
\end{enumerate}
As we already mentioned, to represent the stub-edge, we extend the current 
network $G$ into $G^s$ which contains an extra vertex $s$ of degree $0$ or $1$. 
When the network contains a stub-edge --- connected to vertex  $x$ --- then $G^s$ contains 
the edge $xs$. The vertex $s$ and the stub-edge $xs$ do not belong to $G$.

\medskip

The current network is $G_i$, and we want to construct network 
$G_{i+1}$ which will contain the newly added vertex $w$. The three graphs 
$G_{i},G^s_{i}, G_{i+1}$ completely determine $G^s_{i+1}$; as before, 
$s$ belongs to $G^s_i$ and $G^s_{i+1}$, but not to the networks $G_i$ 
and $G_{i+1}$. Furthermore, $s$ is of degree 0 or 1 in $G^s_i$ and 
$G^s_{i+1}$ (the degree of $s$ in $G^s_i$ and in $G^s_{i+1}$ may be different).

\paragraph{Generalized lifting operation.} Assume $uv \in E(G_i^s)$.
\begin{itemize}
\item If $uv \in E(G_i)$ (therefore $u,v\neq s$), lifting $uv$ to $w$ means
	removing $uv$ and adding $uw,wv$ to the network $G_{i+1}$.
\item If $us$ is an edge in $G^s_i$ and the p-degree of $w$ is even, lifting $us$ to $w$
	means removing $us$ and adding $uw,ws$ to the network $G^s_{i+1}$. 
	(Recall that the vertex $s$ and the stub-edge do not belong to
	$G_{i+1}$.)
\item If $us$ is an edge in $G^s_i$ and the p-degree of $w$ is odd, lifting $us$ to $w$
	means removing $us$ and adding $uw$ to the network $G_{i+1}$. In 
	$G^s_{i+1}$ the degree of $s$ is zero.
\end{itemize}
Now we are ready to introduce the DP-steps. Given a set of edges $M$, the set of
vertices covered by $M$ is denoted by $\cup M$, as per the usual set theoretic
notation.
\paragraph{Admissible DP-steps:}
\begin{enumerate}[{\textit{Op.}\ 1./}]
    \item \textbf{\boldmath If the p-degree of $w$ is $2k$:} select a set $M$ of $k$
        independent edges from $G^s_i$, and lift them all to $w$. If
        the stub-edge $sx$ belongs to $M$ then $d_{i+1}(x)=d_i(x)+1$ in
        $G_{i+1}$, and $d_{i+1}(w)=2k-1$ in $G_{i+1}$; the stub-edge
        $sw$ belongs to $G_{i+1}^s$.
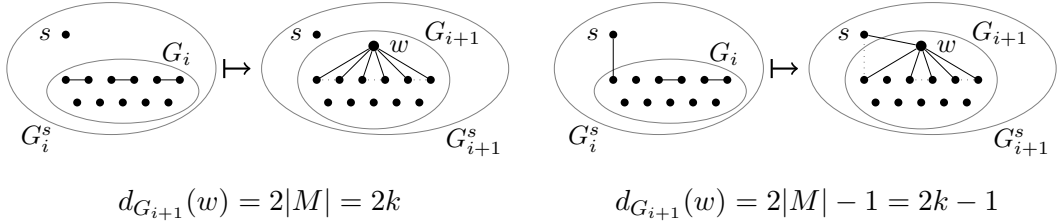
\label{op1}

        \smallskip

        \hspace*{-1.2cm}
        \begin{tikzpicture}[scale=.6]
        \def\a{1}
        \begin{scope}[shift={(0,0)}]
            \node[ellipse,draw=white!50!black,minimum width=2.0cm,minimum
                height=0.85cm,label={[label distance=-0.25cm]30:\small $G_i$}] at (2,0) {};
            \node[ellipse,draw=white!50!black,minimum width=2.75cm,minimum
                height=1.75cm,label={[label distance=-0.15cm]225:\small $G^s_i$}] at
                (1.75,0.5) {};
            \node[fill,circle,minimum size=3pt, inner sep=0,label={west:\small $s$}] at
                (1-0.25,1.25) {};

            \foreach \i in {1,...,5}
                \node[fill,circle,minimum size=3pt, inner sep=0] at
                    (0.5+0.5*\i,-0.25) {};

            \foreach \i in {1,...,3}
                \draw (\i-0.25,0.25) node[fill,circle,minimum size=3pt, inner
                    sep=0] {}--(\i+0.25,0.25) node[fill,circle,minimum size=3pt, inner
                    sep=0] {};

            \draw[thick,|->] (4.2,0.5) -- (4.8,0.5);
            \node at (5,-2.5) {$d_{G_{i+1}}(w)=2|M|=2k$};
        \end{scope}
        \begin{scope}[shift={(5.5,0)}]
            \node[ellipse,draw=white!50!black,minimum width=2.0cm,minimum
                height=1.3cm,label={[label distance=-0.25cm]30:\small $G_{i+1}$}] at
                (2,0.25) {};
            \node[ellipse,draw=white!50!black,minimum width=3.25cm,minimum
                height=1.75cm,label={[label distance=-0.15cm]-45:\small $G^s_{i+1}$}] at
                (2.25,0.5) {};
            \node[fill,circle,minimum size=4pt, inner sep=0,label={east:\small $w$}]
                (w) at (2,1) {};
            \node[fill,circle,minimum size=3pt, inner sep=0,label={west:\small $s$}] at
                (1-0.25,1.25) {};

            \foreach \i in {1,...,5}
                \node[fill,circle,minimum size=3pt, inner sep=0] at
                    (0.5+0.5*\i,-0.25) {};

            \foreach \i in {1,...,3}
            {
                \draw[dotted] (\i-0.25,0.25) node[fill,circle,minimum size=3pt, inner
                    sep=0] {}--(\i+0.25,0.25) node[fill,circle,minimum size=3pt, inner
                    sep=0] {};
                \draw (w) -- (\i-0.25,0.25);
                \draw (w) -- (\i+0.25,0.25);
            }

        \end{scope}

        \begin{scope}[shift={(12,0)}]
            \node[ellipse,draw=white!50!black,minimum width=2.0cm,minimum
                height=0.85cm,label={[label distance=-0.25cm]30:\small $G_i$}] at (2,0) {};
            \node[ellipse,draw=white!50!black,minimum width=2.75cm,minimum
                height=1.75cm,label={[label distance=-0.15cm]225:\small $G^s_i$}] at
                (1.75,0.5) {};
            \node[fill,circle,minimum size=3pt, inner sep=0,label={west:\small
                $s$}] (s) at
                (1-0.25,1.25) {};

            \foreach \i in {1,...,5}
                \node[fill,circle,minimum size=3pt, inner sep=0] at
                    (0.5+0.5*\i,-0.25) {};

            \foreach \i in {1,...,3}
            {
                \ifx\i\a
                    \draw (\i-0.25,0.25) node[fill,circle,minimum size=3pt, inner
                        sep=0] {} -- (s);
                    \node[fill,circle,minimum size=3pt, inner sep=0] at
                          (\i+0.25,0.25) {};
                \else
                    \draw (\i-0.25,0.25) node[fill,circle,minimum size=3pt, inner
                     sep=0] {}--(\i+0.25,0.25) node[fill,circle,minimum size=3pt, inner
                     sep=0] {};
                \fi
            }
            \draw[thick,|->] (4.2,0.5) -- (4.8,0.5);
            \node at (5,-2.5) {$d_{G_{i+1}}(w)=2|M|-1=2k-1$};
        \end{scope}
        \begin{scope}[shift={(12+5.5,0)}]
            \node[ellipse,draw=white!50!black,minimum width=2.0cm,minimum
                height=1.3cm,label={[label distance=-0.25cm]30:\small $G_{i+1}$}] at
                (2,0.25) {};
            \node[ellipse,draw=white!50!black,minimum width=3.25cm,minimum
                height=1.75cm,label={[label distance=-0.15cm]-45:\small $G^s_{i+1}$}] at
                (2.25,0.5) {};
            \node[fill,circle,minimum size=4pt, inner sep=0,label={east:\small $w$}]
                (w) at (2,1) {};
            \node[fill,circle,minimum size=3pt, inner sep=0,label={west:\small
                $s$}] (s) at (1-0.25,1.25) {};

            \foreach \i in {1,...,5}
                \node[fill,circle,minimum size=3pt, inner sep=0] at
                    (0.5+0.5*\i,-0.25) {};

            \foreach \i in {1,...,3}
            {
                \ifx\i\a
                \draw[dotted] (\i-0.25,0.25) node[fill,circle,minimum size=3pt, inner
                        sep=0] {} -- (s);
                    \node[fill,circle,minimum size=3pt, inner sep=0] at
                          (\i+0.25,0.25) {};
                \draw (w) -- (\i-0.25,0.25);
                \draw (w) -- (s);
                \else
                \draw[dotted] (\i-0.25,0.25) node[fill,circle,minimum size=3pt, inner
                    sep=0] {}--(\i+0.25,0.25) node[fill,circle,minimum size=3pt, inner
                    sep=0] {};
                \draw (w) -- (\i-0.25,0.25);
                \draw (w) -- (\i+0.25,0.25);
                \fi
            }
        \end{scope}
    \end{tikzpicture}

    \item \textbf{\boldmath If the p-degree of $w$ is $2k+1$ and $d(s)= 1$ in $G^s_i$:}
        select a set $M$ of $k+1$ independent edges from $G^s_i$, such
        that the stub-edge is included in $M$. Lift every edge of $M$ to
        $w$. In the $G^s_{i+1}$ we have $d(s)=0$.
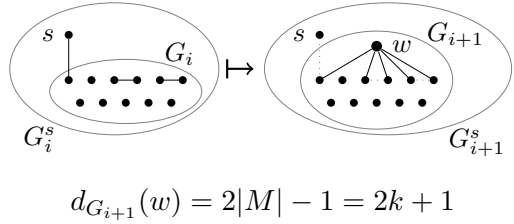
\label{op2}

        \smallskip

        \hspace*{2.4cm}
        \begin{tikzpicture}[scale=.6]
        \def\a{1}
        \begin{scope}[shift={(0,-5)}]
            \node[ellipse,draw=white!50!black,minimum width=2.0cm,minimum
                height=0.85cm,label={[label distance=-0.25cm]30:\small $G_i$}] at (2,0) {};
            \node[ellipse,draw=white!50!black,minimum width=2.75cm,minimum
                height=1.75cm,label={[label distance=-0.15cm]225:\small $G^s_i$}] at
                (1.75,0.5) {};
            \node[fill,circle,minimum size=3pt, inner sep=0,label={west:\small
                $s$}] (s) at
                (1-0.25,1.25) {};

            \foreach \i in {1,...,5}
                \node[fill,circle,minimum size=3pt, inner sep=0] at
                    (0.5+0.5*\i,-0.25) {};

            \foreach \i in {1,...,3}
            {
                \ifx\i\a
                    \draw (\i-0.25,0.25) node[fill,circle,minimum size=3pt, inner
                        sep=0] {} -- (s);
                    \node[fill,circle,minimum size=3pt, inner sep=0] at
                          (\i+0.25,0.25) {};
                \else
                    \draw (\i-0.25,0.25) node[fill,circle,minimum size=3pt, inner
                     sep=0] {}--(\i+0.25,0.25) node[fill,circle,minimum size=3pt, inner
                     sep=0] {};
                \fi
            }
            \draw[thick,|->] (4.2,0.5) -- (4.8,0.5);
            \node at (5,-2.5) {$d_{G_{i+1}}(w)=2|M|-1=2k+1$};
        \end{scope}
        \begin{scope}[shift={(5.5,-5)}]
            \node[ellipse,draw=white!50!black,minimum width=2.0cm,minimum
                height=1.3cm,label={[label distance=-0.25cm]30:\small $G_{i+1}$}] at
                (2,0.25) {};
            \node[ellipse,draw=white!50!black,minimum width=3.25cm,minimum
                height=1.75cm,label={[label distance=-0.15cm]-45:\small $G^s_{i+1}$}] at
                (2.25,0.5) {};
            \node[fill,circle,minimum size=4pt, inner sep=0,label={east:\small $w$}]
                (w) at (2,1) {};
            \node[fill,circle,minimum size=3pt, inner sep=0,label={west:\small
                $s$}] (s) at (1-0.25,1.25) {};

            \foreach \i in {1,...,5}
                \node[fill,circle,minimum size=3pt, inner sep=0] at
                    (0.5+0.5*\i,-0.25) {};

            \foreach \i in {1,...,3}
            {
                \ifx\i\a
                \draw[dotted] (\i-0.25,0.25) node[fill,circle,minimum size=3pt, inner
                        sep=0] {} -- (s);
                    \node[fill,circle,minimum size=3pt, inner sep=0] at
                          (\i+0.25,0.25) {};
                \draw (w) -- (\i-0.25,0.25);
                \else
                \draw[dotted] (\i-0.25,0.25) node[fill,circle,minimum size=3pt, inner
                    sep=0] {}--(\i+0.25,0.25) node[fill,circle,minimum size=3pt, inner
                    sep=0] {};
                \draw (w) -- (\i-0.25,0.25);
                \draw (w) -- (\i+0.25,0.25);
                \fi
            }
        \end{scope}
    \end{tikzpicture}

    \item \textbf{\boldmath If the p-degree of $w$ is $2k+1$ and $d(s)=0$ in $G^s_i$:}
        choose an integer $r\in [0,2k+2]$.\label{op3}
        \begin{enumerate}
            \item \textbf{\boldmath if $r=0$:} select a set $M$ of $k$
                independent edges from $G_i$, and lift $M$ to $w$. Add
                the stub-edge $ws$ to the network
                $G^s_{i+1}$;\label{op3a}
            \item \textbf{\boldmath if $r\in [1,2k+2]$:} select a set $M$ of
                $k+1$ independent edges from $G_i$, and lift $M$ to $w$.
                Let $u$ be the $r^\text{th}$ vertex in $\cup M$;
                remove $uw$ and add the stub-edge $us$ to the network $G^s_{i+1}$.\label{op3b}
        \end{enumerate}

        
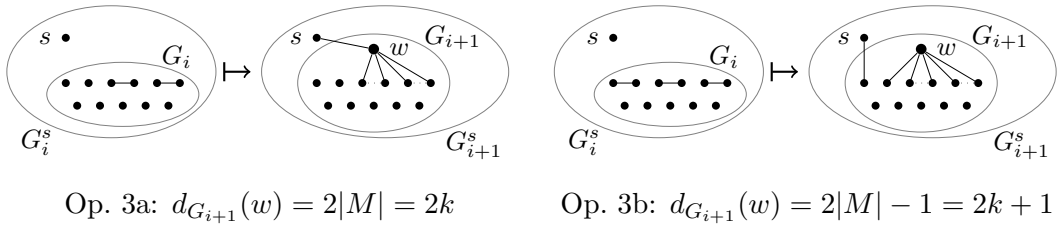
\smallskip

        \hspace*{-1.2cm}
        \begin{tikzpicture}[scale=.6]
        \def\a{1}
        \begin{scope}[shift={(0,-5)}]
            \node[ellipse,draw=white!50!black,minimum width=2.0cm,minimum
                height=0.85cm,label={[label distance=-0.25cm]30:\small $G_i$}] at (2,0) {};
            \node[ellipse,draw=white!50!black,minimum width=2.75cm,minimum
                height=1.75cm,label={[label distance=-0.15cm]225:\small $G^s_i$}] at
                (1.75,0.5) {};
            \node[fill,circle,minimum size=3pt, inner sep=0,label={west:\small
                $s$}] (s) at
                (1-0.25,1.25) {};

            \foreach \i in {1,...,5}
                \node[fill,circle,minimum size=3pt, inner sep=0] at
                    (0.5+0.5*\i,-0.25) {};

            \draw (1-0.25,0.25) node[fill,circle,minimum size=3pt, inner sep=0] {};
            \draw (1+0.25,0.25) node[fill,circle,minimum size=3pt, inner sep=0] {};
            \foreach \i in {2,...,3}
                \draw (\i-0.25,0.25) node[fill,circle,minimum size=3pt, inner
                     sep=0] {}--(\i+0.25,0.25) node[fill,circle,minimum size=3pt, inner sep=0] {};

            \draw[thick,|->] (4.2,0.5) -- (4.8,0.5);
            \node at (5,-2.5) {Op.\ 3a: $d_{G_{i+1}}(w)=2|M|=2k$};
        \end{scope}
        \begin{scope}[shift={(5.5,-5)}]
            \node[ellipse,draw=white!50!black,minimum width=2.0cm,minimum
                height=1.3cm,label={[label distance=-0.25cm]30:\small $G_{i+1}$}] at
                (2,0.25) {};
            \node[ellipse,draw=white!50!black,minimum width=3.25cm,minimum
                height=1.75cm,label={[label distance=-0.15cm]-45:\small $G^s_{i+1}$}] at
                (2.25,0.5) {};
            \node[fill,circle,minimum size=4pt, inner sep=0,label={east:\small $w$}]
                (w) at (2,1) {};
            \node[fill,circle,minimum size=3pt, inner sep=0,label={west:\small
                $s$}] (s) at (1-0.25,1.25) {};
            \draw (w) -- (s);

            \foreach \i in {1,...,5}
                \node[fill,circle,minimum size=3pt, inner sep=0] at
                    (0.5+0.5*\i,-0.25) {};

            \draw (1-0.25,0.25) node[fill,circle,minimum size=3pt, inner sep=0] {};
            \draw (1+0.25,0.25) node[fill,circle,minimum size=3pt, inner sep=0] {};
            \foreach \i in {2,...,3}
            {
                \draw[dotted] (\i-0.25,0.25) node[fill,circle,minimum size=3pt, inner
                    sep=0] {}--(\i+0.25,0.25) node[fill,circle,minimum size=3pt, inner
                    sep=0] {};
                \draw (w) -- (\i-0.25,0.25);
                \draw (w) -- (\i+0.25,0.25);
            }
        \end{scope}

        \begin{scope}[shift={(12,-5)}]
            \node[ellipse,draw=white!50!black,minimum width=2.0cm,minimum
                height=0.85cm,label={[label distance=-0.25cm]30:\small $G_i$}] at (2,0) {};
            \node[ellipse,draw=white!50!black,minimum width=2.75cm,minimum
                height=1.75cm,label={[label distance=-0.15cm]225:\small $G^s_i$}] at
                (1.75,0.5) {};
            \node[fill,circle,minimum size=3pt, inner sep=0,label={west:\small
                $s$}] (s) at
                (1-0.25,1.25) {};

            \foreach \i in {1,...,5}
                \node[fill,circle,minimum size=3pt, inner sep=0] at
                    (0.5+0.5*\i,-0.25) {};

            \foreach \i in {1,...,3}
                \draw (\i-0.25,0.25) node[fill,circle,minimum size=3pt, inner
                     sep=0] {}--(\i+0.25,0.25) node[fill,circle,minimum size=3pt, inner
                     sep=0] {};

            \draw[thick,|->] (4.2,0.5) -- (4.8,0.5);
            \node at (5,-2.5) {Op.\ 3b: $d_{G_{i+1}}(w)=2|M|-1=2k+1$};
        \end{scope}
        \begin{scope}[shift={(12+5.5,-5)}]
            \node[ellipse,draw=white!50!black,minimum width=2.0cm,minimum
                height=1.3cm,label={[label distance=-0.25cm]30:\small $G_{i+1}$}] at
                (2,0.25) {};
            \node[ellipse,draw=white!50!black,minimum width=3.25cm,minimum
                height=1.75cm,label={[label distance=-0.15cm]-45:\small $G^s_{i+1}$}] at
                (2.25,0.5) {};
            \node[fill,circle,minimum size=4pt, inner sep=0,label={east:\small $w$}]
                (w) at (2,1) {};
            \node[fill,circle,minimum size=3pt, inner sep=0,label={west:\small
                $s$}] (s) at (1-0.25,1.25) {};

            \foreach \i in {1,...,5}
                \node[fill,circle,minimum size=3pt, inner sep=0] at
                    (0.5+0.5*\i,-0.25) {};

            \foreach \i in {1,...,3}
            {
                \ifx\i\a
                \node[fill,circle,minimum size=3pt, inner sep=0] at
                     (\i-0.25,0.25) {};
                \node[fill,circle,minimum size=3pt, inner sep=0] at
                    (\i+0.25,0.25) {};
                \draw (s) -- (\i-0.25,0.25);
                \else
                \draw[dotted] (\i-0.25,0.25) node[fill,circle,minimum size=3pt, inner
                    sep=0] {}--(\i+0.25,0.25) node[fill,circle,minimum size=3pt, inner
                    sep=0] {};
                \draw (w) -- (\i-0.25,0.25);
                \fi
                \draw (w) -- (\i+0.25,0.25);
            }
        \end{scope}
    \end{tikzpicture}
\end{enumerate}
Next we introduce the inverse operations of the above defined admissible DP-steps.

\paragraph{Admissible DP-removal steps:}
We want to remove $w$ from $G_{i+1}$ by the inverse of a DP-step.
\begin{enumerate}[{\textit{InvOp.} 1./}]
\item \textbf{\boldmath If $d(w)=2k$ in the network $G^s_{i+1}$:} choose a set
	$M$ of $k$ independent non-edges in the $G^s_{i+1}$-neighborhood of $w$.
	Change the non-edges in $M$ to edges and remove $w$ along with its incident edges to obtain $G^s_i$. This
	defines the network $G_i$, as well. If $s$ is a neighbor of $w$ in $G_{i+1}^s$, then the non-edge covering $s$ in $M$ becomes the stub-edge after the inverse step.
	This is the inverse of Op.~\ref{op1}.\label{invop1}

\item \textbf{\boldmath If $d(w)=2k+1$ and $d(s)=0$ in $G^s_{i+1}$:} select
	a set $M$ of $k$ independent non-edges in the $G_{i+1}$-neighborhood of
	$w$. Denote by $x$ the vertex connected to $w$ in $G_{i+1}$ which is not
	covered by $M$. Change the non-edges in $M$ to edges and remove $w$ along with its edges to obtain
	$G^s_i$. Remove $w$ along with its edges. Add the stub-edge $xs$ to $G^s_i$. 
	This is the inverse of Op.~\ref{op2}.\label{invop2}

\item[{\textit{InvOp.} 3a./}] \textbf{\boldmath If $d(w)=2k+1$ and the stub-edge $ws$ 
    belongs to  $G^s_{i+1}$:} select a set $M$ of $k$  independent non-edges in
	the $G_{i+1}$-neighborhood of $w$. Change the non-edges in $M$ to edges
	to obtain $G^s_i$. Remove $w$ along with its edges. Then $d(s)=0$ 
	in $G^s_i$. This is the inverse of Op.~\ref{op3a}.\label{invop3a}

\item[{\textit{InvOp.} 3b./}] \textbf{\boldmath If $d(w)=2k+1$ and for the stub-edge $us$ in the
	$G^s_{i+1}$ we have $u \ne w$ and $uw$ is not an edge in $G_{i+1}$:}
	select a set $M$ of $k+1$ independent non-edges from
	$G^s_{i+1}[\Gamma_{G^s_{i+1}}(w)\cup\{u\}]$ (recall that $\Gamma_G(v)$ denotes the 
	neighborhood graph of vertex $v$ in graph $G$). Change the non-edges in $M$
	to edges to obtain $G^s_i$. Remove $w$ along with its edges. In
	the network $G^s_i$ we have degree$(s)=0$. This is the inverse of
	Op.~\ref{op3b}.\label{invop3b}
\end{enumerate}

\begin{definition}[\textbf{Irreducibility}]
    A graph $G$ is called \emph{irreducible} if none of the above inverse
    operations can be applied to any vertex $w$ of $G$.
\end{definition}

\paragraph{Two simple examples.}\label{sec:prop}
One can ask whether a small irreducible ``kernel'' network to which a given
network can be reduced to with DP-removals is unique. The answer, not
surprisingly, is negative.

\begin{example}
\Cref{fig:2K4} depicts a series of DP-steps and DP-removals that lead from
$2K_4$ to a $K_4$.
\end{example}

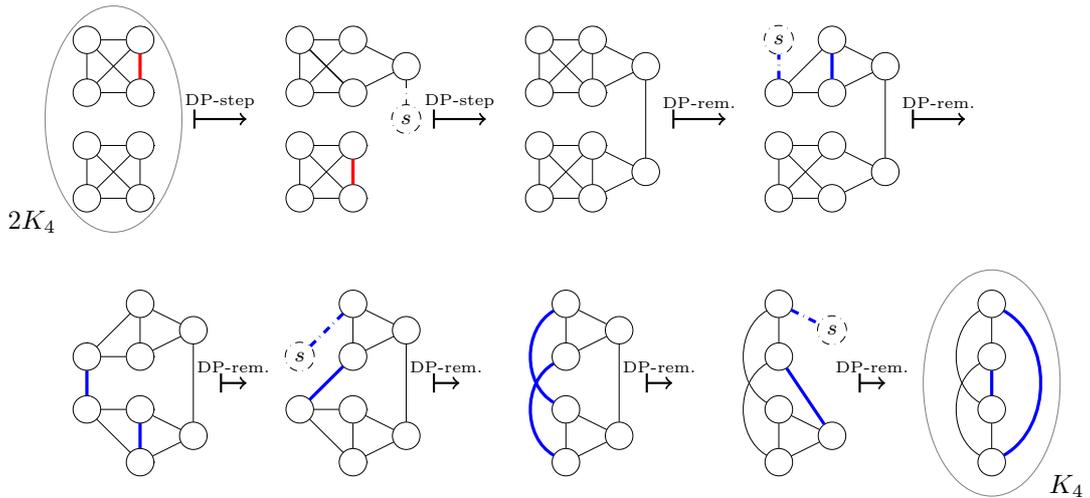
\begin{figure}[H]
    \centering
    \begin{tikzpicture}[scale=.7]
    \begin{scope}[shift={(-3,0)}]
        \node[draw,circle] (u1) at (0,0) {};
        \node[draw,circle] (u2) at (1,0) {};
        \node[draw,circle] (u3) at (1,1) {};
        \node[draw,circle] (u4) at (0,1) {};
        \node[draw,circle] (v1) at (0,2) {};
        \node[draw,circle] (v2) at (1,2) {};
        \node[draw,circle] (v3) at (1,3) {};
        \node[draw,circle] (v4) at (0,3) {};

        \foreach \i in {1,2,3}
        \foreach \j in {\i,...,4}
        {
            \draw (u\i)--(u\j);
            \draw (v\i)--(v\j);
        }
        \draw[red, very thick] (v2)--(v3);
        \draw[thick,|->] (2.0,1.5) -- node[midway,above] {\tiny DP-step} (3.0,1.5);
        \node[ellipse,draw=white!50!black,minimum width=1.8cm,minimum
                height=3cm,label={[label distance=0.0cm]240:\small $2K_4$}] at
                (0.5,1.5) {};
    \end{scope}
    \begin{scope}[shift={(1,0)}]
        \node[draw,circle] (u1) at (0,0) {};
        \node[draw,circle] (u2) at (1,0) {};
        \node[draw,circle] (u3) at (1,1) {};
        \node[draw,circle] (u4) at (0,1) {};
        \node[draw,circle] (v1) at (0,2) {};
        \node[draw,circle] (v2) at (1,2) {};
        \node[draw,circle] (v3) at (1,3) {};
        \node[draw,circle] (v4) at (0,3) {};

        \foreach \i in {1,2,3}
        \foreach \j in {\i,...,4}
        {
            \draw (u\i)--(u\j);
            \ifthenelse{\i=2}{\draw (v2) -- (v4);}{\draw (v\i)--(v\j);}
        }

        \draw[red,very thick] (u2)--(u3);
        \node[draw,circle] (x)  at (2.0,2.5) {};
        \draw (x) -- (v2);
        \draw (x) -- (v3);
        \node[draw,circle,dashdotted, inner sep=2] (s) at ($ (x)-(0,1) $)
            {\scriptsize $s$};
        \draw[dashdotted] (x) -- (s);
        \draw[thick,|->] (2.5,1.5) -- node[midway,above] {\tiny DP-step} (3.5,1.5);
    \end{scope}
    \begin{scope}[shift={(5.5,0)}]
        \node[draw,circle] (u1) at (0,0) {};
        \node[draw,circle] (u2) at (1,0) {};
        \node[draw,circle] (u3) at (1,1) {};
        \node[draw,circle] (u4) at (0,1) {};
        \node[draw,circle] (v1) at (0,2) {};
        \node[draw,circle] (v2) at (1,2) {};
        \node[draw,circle] (v3) at (1,3) {};
        \node[draw,circle] (v4) at (0,3) {};

        \foreach \i in {1,2,3}
        \foreach \j in {\i,...,4}
        {
            \draw (u\i)--(u\j);
            \draw (v\i)--(v\j);
        }

        \node[draw,circle] (x)  at (2.0,2.5) {};
        \draw (x) -- (v2);
        \draw (x) -- (v3);

        \node[draw,circle] (y)  at (2.0,0.5) {};
        \draw (y) -- (u2);
        \draw (y) -- (u3);
        \draw (x) -- (y);
        \draw[thick,|->] (2.5,1.5) -- node[midway,above] {\tiny DP-rem.} (3.5,1.5);
    \end{scope}
    \begin{scope}[shift={(10,0)}]
        \node[draw,circle] (u1) at (0,0) {};
        \node[draw,circle] (u2) at (1,0) {};
        \node[draw,circle] (u3) at (1,1) {};
        \node[draw,circle] (u4) at (0,1) {};
        \node[draw,circle] (v1) at (0,2) {};
        \node[draw,circle] (v2) at (1,2) {};
        \node[draw,circle] (v3) at (1,3) {};

        \draw (u1) -- (u4);
        \draw (u2) -- (u4);
        \draw (u3) -- (u4);
        \draw (u3) -- (u1);
        \draw (u1) -- (u2);

        \draw (v1) -- (v2);
        \draw (v1) -- (v3);
        \node[draw,circle,dashdotted, inner sep=2] (s) at ($ (v1)+(0,1) $)
            {\scriptsize $s$};
        \draw[blue,dashdotted, very thick] (v1) -- (s);

        \draw[blue, very thick] (v2)--(v3);

        \node[draw,circle] (x)  at (2.0,2.5) {};
        \draw (x) -- (v2);
        \draw (x) -- (v3);

        \node[draw,circle] (y)  at (2.0,0.5) {};
        \draw (y) -- (u2);
        \draw (y) -- (u3);
        \draw (x) -- (y);
        \draw[thick,|->] (2.5,1.5) -- node[midway,above] {\tiny DP-rem.} (3.5,1.5);
    \end{scope}
    \begin{scope}[shift={(-3,-5)}]
        \node[draw,circle] (u2) at (1,0) {};
        \node[draw,circle] (u3) at (1,1) {};
        \node[draw,circle] (u4) at (0,1) {};
        \node[draw,circle] (v1) at (0,2) {};
        \node[draw,circle] (v2) at (1,2) {};
        \node[draw,circle] (v3) at (1,3) {};

        \foreach \i in {1,2}
        \foreach \j in {\i,...,3}
        {
            \draw (v\i)--(v\j);
        }
        \foreach \i in {2,3}
        \foreach \j in {\i,...,4}
        {
            \draw (u\i)--(u\j);
        }
        \draw[blue, very thick] (v1) -- (u4);

        \draw[blue, very thick] (u2)--(u3);

        \node[draw,circle] (x)  at (2.0,2.5) {};
        \draw (x) -- (v2);
        \draw (x) -- (v3);

        \node[draw,circle] (y)  at (2.0,0.5) {};
        \draw (y) -- (u2);
        \draw (y) -- (u3);
        \draw (x) -- (y);
        \draw[thick,|->] (2.5,1.5) -- node[midway,above] {\tiny DP-rem.} (3.0,1.5);
    \end{scope}
    \begin{scope}[shift={(1,-5)}]
        \node[draw,circle] (u2) at (1,0) {};
        \node[draw,circle] (u3) at (1,1) {};
        \node[draw,circle] (u4) at (0,1) {};
        \node[draw,circle] (v2) at (1,2) {};
        \node[draw,circle] (v3) at (1,3) {};

        \draw (v2) -- (v3);
        \foreach \i in {2,3}
        \foreach \j in {\i,...,4}
        {
            \draw (u\i)--(u\j);
        }

        \node[draw,circle] (x)  at (2.0,2.5) {};
        \draw (x) -- (v2);
        \draw (x) -- (v3);

        \node[draw,circle] (y)  at (2.0,0.5) {};
        \draw (y) -- (u2);
        \draw (y) -- (u3);
        \draw (x) -- (y);

        \draw[blue, very thick] (u4)--(v2);
        \node[draw,circle,dashdotted, inner sep=2] (s) at ($ (v3)-(1,1) $)
            {\scriptsize $s$};
        \draw[dashdotted,blue, very thick] (v3) -- (s);
        \draw[thick,|->] (2.5,1.5) -- node[midway,above] {\tiny DP-rem.} (3.0,1.5);
    \end{scope}
    \begin{scope}[shift={(5,-5)}]
        \node[draw,circle] (u2) at (1,0) {};
        \node[draw,circle] (u3) at (1,1) {};
        \node[draw,circle] (v2) at (1,2) {};
        \node[draw,circle] (v3) at (1,3) {};

        \draw (v2) -- (v3);
        \draw (u2) -- (u3);

        \node[draw,circle] (x)  at (2.0,2.5) {};
        \draw (x) -- (v2);
        \draw (x) -- (v3);

        \node[draw,circle] (y)  at (2.0,0.5) {};
        \draw (y) -- (u2);
        \draw (y) -- (u3);
        \draw (x) -- (y);

        \draw[blue, very thick] (u2) edge[bend left=60] (v2);
        \draw[blue, very thick] (u3) edge[bend left=60] (v3);
        \draw[thick,|->] (2.5,1.5) -- node[midway,above] {\tiny DP-rem.} (3.0,1.5);
    \end{scope}
    \begin{scope}[shift={(9,-5)}]

        \node[draw,circle] (u2) at (1,0) {};
        \node[draw,circle] (u3) at (1,1) {};
        \node[draw,circle] (v2) at (1,2) {};
        \node[draw,circle] (v3) at (1,3) {};

        \draw (v2) -- (v3);
        \draw (u2) -- (u3);

        \node[draw,circle] (y)  at (2.0,0.5) {};
        \draw (y) -- (u2);
        \draw (y) -- (u3);

        \draw (u2) edge[bend left=60] (v2);
        \draw (u3) edge[bend left=60] (v3);

        \draw[blue, very thick] (v2) -- (y);
        \node[draw,circle,dashdotted, inner sep=2] (s) at (2.0,2.5)
            {\scriptsize $s$};
        \draw[dashdotted,blue, very thick] (v3) -- (s);
        \draw[thick,|->] (2.5,1.5) -- node[midway,above] {\tiny DP-rem.} (3.0,1.5);
    \end{scope}
    \begin{scope}[shift={(13,-5)}]
        \node[draw,circle] (u2) at (1,0) {};
        \node[draw,circle] (u3) at (1,1) {};
        \node[draw,circle] (v2) at (1,2) {};
        \node[draw,circle] (v3) at (1,3) {};

        \draw (v2) -- (v3);
        \draw (u2) -- (u3);

        \draw (u2) edge[bend left=60] (v2);
        \draw (u3) edge[bend left=60] (v3);

        \draw[blue, very thick] (v2) -- (u3);
        \draw[blue, very thick] (v3) edge[bend left=60] (u2);
        \node[ellipse,draw=white!50!black,minimum width=1.8cm,minimum
                height=3cm,label={[label distance=0.0cm]300:\small $K_4$}] at
                (1.0,1.5) {};
    \end{scope}
    \end{tikzpicture}
    \caption{2 DP-steps operations followed by 6 DP-removals
        transforms $2K_4$ into $K_4$. \textcolor{red}{Red edges}: to be
        removed by DP steps; \textcolor{blue}{blue edges}: new edges
        created by inverse DP steps. Dash-dotted vertices and edges represent
        the stub vertex and the stub edge.}\label{fig:2K4}
\end{figure}

A natural question to ask is: How difficult is it to find irreducible networks? 
One simple example is the complete graph, in which, clearly, no inverse DP-step can be
performed. Intuitively, very dense networks are irreducible. Does the converse
that networks that are sparse enough are not irreducible hold? Again, the
answer is negative, as the following example shows.

\begin{example}\label{ex:K4}
    For any $n\in\mathbb{N}^+$, there exists an irreducible 4-regular graph on
    $4n$ vertices which is connected and vertex-transitive,
    see~\Cref{fig:irreduc}.
\end{example}
\begin{figure}[ht]
    \centering
    \begin{tikzpicture}[scale=.5]
    \def\n{8}
    \foreach \i in {1,...,\n}
    {
        \foreach \j in {1,...,4}
        {
            \node[fill,circle,inner sep=2pt] (v\i\j) at ($ (\i*360/\n:4) +
                (\i*360/\n+45+90*\j:0.75) $) {};
        }
        \draw (v\i1) -- (v\i2) -- (v\i3) -- (v\i4) -- (v\i2);
        \draw (v\i3) -- (v\i1) -- (v\i4);
    }
    \foreach \i in {1,...,\n}
    {
        \pgfmathtruncatemacro{\j}{mod(\i,\n)+1}
        \draw (v\i1) -- (v\j2);
        \draw (v\i4) -- (v\j3);
    }
    \end{tikzpicture}
    \caption{A graph which is irreducible.}\label{fig:irreduc}
\end{figure}
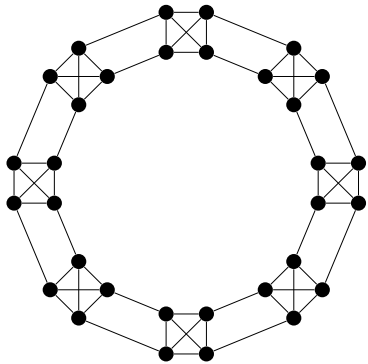

There is a polynomial time algorithm to decide whether a graph admits a DP-removal or not: one can run Edmonds' blossom algorithm for the complement of the neighborhood graph $\Gamma(v)$ of each vertex $v\in V(G)$  (\cite{edmonds}). Nevertheless, we cannot expect to characterize removability of even a set of size $n^\varepsilon$, because this problem is \textsc{NP}-hard, as shown in \Cref{sec:NP}.

\section{Particular DPG models}\label{sec:kind}
Next, we list a number of examples of different kinds of DPG dynamics. Some examples
of this section are partially based on the paper~\cite{DPG21}, where particular
cases of the following results and derivations have already been given in
that paper's Supplementary Information.

\medskip

The two freedoms in designing a DPG model are the size of the next (proper) incoming degree and how the matching of the appropriate size is selected. The size of the matching and the inserted degree constrain one another during the process. There are several possibilities for finding a matching of a given size. One can, for example, choose it greedily. Or, one can seek a maximum size matching and take a random subset of it of the needed size. Alternatively, one can try to choose uniformly and randomly one matching from the set of all the matchings with the predefined size. The relative advantages and disadvantages should also be studied from an algorithmic point of view and it is beyond the scope of this paper.

\subsection{Linear DPG}\label{sec:LinearDPG}

Denote by $\nu(G)$ the matching number of the graph $G$. Let $0<c \le 1$ be a
constant. In the \textbf{linear DPG} model, the incoming proper degrees are defined via 
$\lceil 2c\nu(G_i) \rceil$. 
The simulations show that for any value of $c$
typically there exists a very large matching (close to perfect) in $G_i$.
Therefore the degree sequence of the incoming vertex is linear in $i$ and thus also
the cumulative distribution of the degrees of $G_i$. Next, we
will study the matching number analytically for linear DPG processes.

\medskip

To analyze the DPG process, we need tools to estimate the matching number. Since the set of neighbors of an incoming vertex is only restricted by the set of available matchings, which is difficult to track, it is natural to try to estimate the matching number based only on the degree sequence. There are some available tools for this goal.
Let $\chi'(G)$ and $\Delta(G)$ respectively denote the edge-chromatic number and the maximum degree~of~$G$.

\begin{theorem}[Vizing~\cite{vizing}]\label{thm:vizing}
	$\chi'(G)\le \Delta(G)+1$ holds for any simple graph $G$.
\end{theorem}
From Vizing's theorem, one can easily conclude that
\begin{equation*}\label{eq:vizing}
\nu(G) \ge \frac{|E(G)|}{\Delta(G)+1}.
\end{equation*}
This can be a tight bound in case the degree distribution is concentrated, but in the case of a wide range of degrees, this may be very far from being sharp.

\begin{theorem}[Pósa, 1962~\cite{posa}]\label{thm:posa}
    Let $G$ be a graph on $n$ vertices. Suppose its degree sequence
    ${(d(k))}_{k=1}^n$ is in increasing order. If for every $1\le k<n/2$ we have
    $(k+1)\le d(k)$, then $G$ is Hamiltonian.
\end{theorem}

When the graph is dense and there are not many low degree vertices,~\cite[Theorem~S5]{DPG21} (in the Supplementary Information) provides a tighter bound on the matching number, based on \Cref{thm:posa}.
\begin{definition}\label{def:tGq}
    Denote by $D_{\le q}(G)$ the number of vertices of $G$ whose degree does not
    exceed $q$. (This quantity is denoted by $t_{G}(q)$ in~\cite{DPG21}.)
\end{definition}
\begin{theorem}[{\cite[Theorem~S5]{DPG21}}]
Let $G$ be a simple graph on $n$ vertices. Let
\begin{equation*} \label{eq:rG1}
    r(G):=\min \!\left\{ \ell \in \mathbb{Z}^{+}\ :\ \max_{0\le q<\frac{n-\ell}{2}}\left(D_{\le q}(G)-q+1\right)\le \ell \right\}
\end{equation*}
then $G$ has a matching of size: $\left\lceil{\frac{n-r(G)}{2}}\right \rceil \leq \nu(G)$.
\end{theorem}
This result was proved already, albeit in a slightly different form, by Bondy and Chv\'atal in 1976 (see \cite[Theorem 5.1]{BC76}).
The next result appeared in essence in~\cite{DPG21}, albeit in a slightly different form.
\begin{corollary}\label{lemma:rnu}
    For a simple graph $G$ on $n$ vertices
    \begin{equation*}
        \nu(G)\ge \min_{0\le q< \frac{n-1}{2}}\max\left ( \frac12\left(n-D_{\le q}(G)+q-1\right), q \right).
    \end{equation*}
\end{corollary}
\begin{proof}
It is sufficient to show that
\begin{equation}\label{eq:rsimple}
r(G) =\max_{0\le q< \frac{n-1}{2}}     \min\left( D_{\le q}(G)-q+1, n-2q \right).
\end{equation}
While $0\le q<\frac{n-r(G)}{2}$, we have $n-2q>r(G)$, and for $q\ge \frac{n-r(G)}{2}$, we have $n-2q\le r(G)$, thus the right hand side of~\eqref{eq:rsimple} is at most $r(G)$. If for some $0\le q<\frac{n-r(G)}{2}$ we have $D_{\le q}(G)-q+1=r(G)$, then the right hand side is equal to $r(G)$.

    \medskip

The right hand side of~\eqref{eq:rsimple} is trivially $\ge 1$, so if $r(G)=1$, the lemma holds.

    \medskip

Suppose that $r(G)>1$ and for all $0\le q<\frac{n-r(G)}{2}$ we have $D_{\le
q}(G)-q+1<r(G)$. Let $\mu=\frac{n-r(G)}{2}$. By the minimality of $r(G)$, this implies that we must have $n\equiv r\pmod{2}$ and
    \begin{align*}
        D_{\le \mu}(G)-\mu+1\ge r(G).
    \end{align*}
By substituting $q=\mu=\frac{n-r(G)}{2}$, the right hand side of~\eqref{eq:rsimple} is at least $r(G)$.
\end{proof}

When the constant $c=1$ and even degrees are added only, the \textbf{linear DPG} model is called the \textbf{MaxDPG} model.  Using the previous estimations, paper~\cite{DPG21} managed to prove the following first-order estimate of growth for the MaxDPG process:
\begin{theorem}[\cite{DPG21}]\label{th:Max}
    Let the MaxDPG process produce the network series ${(G_n)}_{n=n_0}^\infty$ from the initial network $G_{n_0}$ (which has at least one edge). Then for large enough $n$ we have
    \begin{equation*}\label{eq:firstorder}
        d(v_n) \ge n -2 \log_2 n - \mathcal{O}(1).
    \end{equation*}
\end{theorem}
(Instead of proving this statement directly, we will prove a generalization of \Cref{th:Max} in
the following \Cref{lemma:linear_lower_bound}.) The edge density of $G_n$ is
$\rho_n=\frac{1}{2}- \mathcal{O} ({\log_2{n}}/{n})$, and one
can show~\cite{DPG21} that it has a core-periphery structure. More precisely, it resembles a split graph in which the
nodes are partitioned into three classes: one inducing a clique in $G_n$, another is an
independent set, and the third set contains at most $\mathcal{O}(\log_2 n)$ vertices.

\medskip

There are many real-life situations where the networks have a well-defined
\emph{core–periphery structure}: such a structure consists of a well-connected
core and a periphery that is connected to the core but sparsely connected
internally. Therefore our discussion above shows that the MaxDPG dynamics can
provide random (however, not necessarily uniformly random) examples of core-periphery networks.

\medskip

Let us now study an extension of \Cref{th:Max} for other values of $c$. We
assume that the process so far has produced the network $G_{n_0}$.
\begin{lemma}\label{lemma:linear_lower_bound}
    Let $\frac12<c\le 1$, let constant $K\ge 0$ and suppose that $d(i)\ge
    (2c-1)i-K$ holds for $1\le i\le n_0$, where $d$ is the degree sequence of
    the graph $G_{n_0}$. Then the degree of the vertices inserted
    iteratively into $G_{n_0}$ by a linear DPG process (with the multiplicative
    constant $c$) satisfy
\begin{equation*}
    d(n)\ge (2c-1)n-K-2 \quad\hbox{\ for }\quad n_0<n\le 2c\left(n_0+1-\frac{3} {2c-1}\right)-K.
\end{equation*}
\end{lemma}
\begin{proof}
Let us estimate $n-D_{\le q}(G)$ from below. Let us assume that
    \begin{equation}\label{eq:valid1}
        (2c-1)(n_0+1)-K-2\ge q+1.
    \end{equation}
Note that
    \begin{equation*}
        (2c-1)i-K \ge q+1\Longleftrightarrow  i \ge \frac{q+1+K}{2c-1}
    \end{equation*}
therefore
    \begin{equation}\label{eq:linear_lower_bound}
        n-D_{\le q}(G)\ge n-\frac{q+1+K}{2c-1}.
    \end{equation}
If $(2c-1)n-K-2\ge 0$, i.e., the statement of this lemma is not trivial, then $q<\frac{n-1}{2}$ implies~\eqref{eq:valid1}. Therefore, we can substitute~\eqref{eq:linear_lower_bound} into \Cref{lemma:rnu} to     obtain
    \begin{equation*}
        \nu(G_n)\ge \min_{0\le q< \frac{n-1}{2}}\max\Big(
        \frac12\left(n-\frac{q+1+K}{2c-1}+q-1\right), q \Big).
    \end{equation*}
The right hand side is minimized by the value of $q$ for which the two arguments of $\max$ are equal, since one of the arguments of $\max$ is monotone increasing, while the other is monotone decreasing:
    \begin{equation*}
        q:=\frac{2c-1}{2c}\left(n-1-\frac{K+1}{2c-1}\right).
    \end{equation*}
By the linear DPG rule, we have
    \begin{equation*}
        d(n+1)=2\lceil c\nu(G_n)\rceil\ge (2c-1)\left(n-1-\frac{K+1}{2c-1}\right).
    \end{equation*}
The right hand side meets our wishes if
    \begin{align*}
        (2c-1)\left(n-1-\frac{K+1}{2c-1}\right)&\ge
        (2c-1)(n+1)-K-3\\
        -(K+1)&\ge 2(2c-1)-K-3\\
        2&\ge 2(2c-1)
    \end{align*}
which holds, by definition.
\end{proof}

\begin{corollary}
If the initial $n_0$ is large enough, then in the linear DPG process started
with $G_{n_0}$, as $n\to\infty$, then the degree $d(n)$ of the $n\textsuperscript{th}$ vertex satisfies
    \begin{equation*}
        d(n)\ge (2c-1)n-K-\mathcal{O}(\log_{2c}n),
    \end{equation*}
    or in words, $d(n)$ is linear in $n$.
\end{corollary}
Finding a maximum matching in $G_n$ is a problem solvable in polynomial time, as
demonstrated by Edmonds' blossom algorithm (\cite{edmonds}). However, finding a
\emph{random} maximum matching is a much more complicated problem, since typically
there are exponentially many maximum matchings in a dense graph. In the next part of this subsection we discuss a number of ways to deal with this problem.

\subsubsection*{Heuristics - How to find a random maximum matching?}
For a bipartite network, it is possible to quickly find a uniformly random maximum
matching (actually, this can be extended to finding a random matching of given
cardinality). This is based Jerrum and Sinclair's Markov chain method (see, for
example,~\cite{JS90}). They consider the set $\P$ of all perfect matchings in
$G$, and the set $\N(u,v)$ of all almost perfect matchings in $G$ that do not
cover the vertices $u,v$ (these vertices are called the \textbf{hole}s in the graph).
They consider the following algorithm: let $M$ be a perfect or almost perfect
matching in $G$.
\begin{enumerate}[{\rm (JS1)}]
\item If $M \in \P$, randomly choose an edge $e \in M$ and make the transition to $M \setminus \{e\}$.
\item If $M \in \N(u, v)$, randomly choose a vertex $x \in V$. If $x \in \{u, v\}$ and $u$ is adjacent to $v$, make the transition to $M \cup \{(u,v)\}\in \P$. Otherwise, let $y \in V$ be the vertex matched with $x$ in $M$, and randomly choose $w \in \{u, v\}$. If $x$ is adjacent to $w$, make the transition to the matching $M \setminus \{(x, y)\}\cup \{(x,w)\}\in \N(u,y)$.
\end{enumerate}
Jerrum and Sinclair proved that this Markov chain is fast mixing, so it will
find an almost uniform sample of sets of maximum matchings in polynomial time,
if the graph $G$ is bipartite. The method can be extended to edge-weighted
bipartite graphs (by Jerrum, Sinclair and Vigoda,~\cite{JSV04}). Unfortunately,
this method cannot be extended to graphs in general, as Štefankovič, Vigoda and
Wilmes proved in~\cite{SVW18}. They found graph classes with the following
property: the network $G$ has a large number of perfect matchings, however,
there are holes in $G$ such that the number of almost perfect matching with this
hole is constant. Fortunately, in the case of linear DPG processes, the actual
networks $G_n$ seem to be very far from this disadvantageous situation. In
fact, the symptomatic example  in~\cite{SVW18} is very close to the one depicted
in \Cref{fig:irreduc}. Therefore the JS chain provides a good candidate for
heuristics in the case of linear DPG process.

It is plausible that using genuinely uniform random maximum matchings, the
growth rate of the matching number will be different from the case when the
process uses a not necessarily randomly selected maximum matching. There is no
known evidence in this regard.

\medskip

Another question is whether the network produced by a DPG process can be
considered \emph{random} conditioned on its degree sequence. If the experimenter
desires a truly random sample from the realizations of the generated degree
sequence, then one can use the switch Markov chain to find such random example
from the initially generated network (see, for example,~\cite{P-stable}). The
switch Markov chain is known to provide high-quality random samples if the
degree sequence is $P$-stable (see \cite{P-stable}).

\subsection{Scale-free DPG}\label{sec:scale}
As demonstrated empirically in~\cite{DPG21}, the DPG process can also be used to generate real-world
like synthetic scale-free networks in such a way that the process does not inherently prefer
any vertex over another, and the degrees of already inserted vertices do not change.  
Moreover, simulations showed~\cite{DPG21} that the
generated degree sequences are indeed scale-free with the desired exponent. 
In this section we discuss the protocol in detail and prove that the generated 
degree sequence belongs to the set
of \emph{power-law distribution-bounded} degree sequences.
Wormald and Gao (2016,~\cite{GW16}) introduced this class of scale-free degree
sequences, because most real-world networks do not obey the more traditional
density-bounded power-law. 
\begin{definition}[\cite{GW16}]\label{def:powerlaw}
Let $D_i(G)$ be the set of vertices with degree $i$ in $G$. Similarly, let
$D_{\ge i}(G)$ be the number of vertices with degree greater or equal to $i$ in
$G$. Then the degree sequence of $G$ is 
\begin{itemize}
\item \textbf{power-law density-bounded} with parameters $\gamma$ and $C$, if for all $i\in[1,n]$,
			 \begin{equation*}\label{eq:powerlawdensity}
				D_i(G)\le Cni^{-\gamma}
			 \end{equation*}
\item \textbf{power-law distribution-bounded} with parameters $\gamma$ and $C$, if for all $i\in[1,n]$
			\begin{equation}\label{eq:powerlawdistrib}
                D_{\ge i}(G)=\sum_{j=i}^n D_i(G)\le \sum_{j=i}^\infty Cnj^{-\gamma}.
			\end{equation}
\end{itemize}
Notice the maximum degree is much smaller in the former class. This is analogous
to the difference between the preferential attachment and the Chung-Lu models.
\end{definition}
Also note that the parameters of a power-law distribution-bounded degree sequence
without isolated vertices satisfy $C\ge 1/{\zeta(\gamma)}$, where $\zeta$ is the Riemann zeta function.

\medskip

\textbf{Scale-free DPG protocol:}\\
\textbf{(SF)} Let $\nu:=\nu(G_n)$. Sample an integer $i$ from the interval
$[1, 2\nu]$ with probability $p_{i} = i^{-\gamma}/\sum_{j=1}^{2\nu} j^{-\gamma}$. Add a vertex of proper degree $i$ to the network via a DP-step.

\medskip

Let $(G_n)_{n=n_0}^\infty$ be generated by
the scale-free DPG protocol. We will show first that the degree
sequence $d(G_n)$ is a distribution-bounded power-law degree sequence with
parameter $\gamma$, with probability exponentially close to $1$ as the function
of a parameter $c$, which we call the level of certainty. We will also compute a bound on the
second parameter $C$, which will depend on both $\gamma$ and $c$. We will show next that
$\Delta(G_n)=\Omega(n^{1/(\gamma-1)})$, thus for large enough~$n$, $d(G_n)$ is
not a density-bounded power-law degree sequence for parameter~$\gamma>2$.

\begin{lemma}\label{lem:sfbound}
For any $\gamma>1$ and $c>0$, SF-DPG generates a sequence of graphs with  distribution-bounded power-law degree sequence with coefficient
    \begin{equation}\label{eq:sfboundC}
        C=\frac{1+\sqrt{c}}{\zeta(\gamma)-\frac{1}{\gamma-1}}
    \end{equation}
with probability which is exponentially close to $1$ (as $c$ increases). For $\gamma\ge 2$ and $c\ge \frac14$, the probability of failure is at most $12\cdot 10^{-6c}$ and $C\le 2(1+\sqrt{c})$.
\end{lemma}
\begin{proof}
The value of the Riemann zeta function $\zeta(\gamma)=\sum_{j=1}^\infty
j^{-\gamma}$ is finite for any $\gamma>1$. Let $\zeta(\gamma,i):=\sum_{j=i}^\infty j^{-\gamma}$. For $i>1$, we have:
\begin{align*}
    \zeta(\gamma,i)&<\int_{i-1}^\infty j^{-\gamma}dj= \frac{1}{\gamma-1}{(i-1)}^{1-\gamma}\\
	\zeta(\gamma,i)&>\int_{i}^\infty j^{-\gamma}dj= \frac{1}{\gamma-1}i^{1-\gamma}
\end{align*}
From these bounds it follows that $C>1$ (take $i=2$). By construction, the number of vertices of degree at least $1$  in $G_n$ is $n$, i.e., none of the vertices are isolated. In order to satisfy  \Cref{eq:powerlawdistrib} for $i=1$, we need
    \begin{equation*}
	    n\le Cn\cdot\zeta(\gamma),
    \end{equation*}
which holds, as both $C\ge 1$ and $\zeta(\gamma)\ge 1$. Recall that $D_i(G_n)$
is the number of vertices of degree $i$ in $G_n$. For $i\ge 2$, the expected number of vertices of degree at least $i$ in $G_n$ is
    \begin{align*}
        \sum_{j=i}^{n-1}\mathbb{E}(D_j(G_n))&\le\sum_{j=i+1}^n
        \frac{\sum_{k=i}^{j-1} k^{-\gamma}}{\zeta(\gamma)-\frac{1}{\gamma-1}j^{1-\gamma}}\le
        \frac{1}{\zeta(\gamma)-\frac{1}{\gamma-1}}\sum_{j=i+1}^n
        \sum_{k=i}^{j-1}k^{-\gamma} \\
	&\le
	\frac{1}{\zeta(\gamma)-\frac{1}{\gamma-1}}\sum_{k=i}^{n-1}(n-k)k^{-\gamma}\le
	\frac{n}{\zeta(\gamma)-\frac{1}{\gamma-1}}\zeta(\gamma,i)
    \end{align*}
    The quantity $D_{\ge i}(G_n)$ can be estimated from above by the sum of $n-i$
independent indicators. Therefore, by Hoeffding's inequality, we have
    \begin{equation*}
        \Pr\big(D_{\ge i}(G_n)-\mathbb{E}(D_{\ge i}(G_n))>t\big)\le e^{-2t^2/(n-i)}\le e^{-2t^2/n}
    \end{equation*}
Substituting $t=\frac{\sqrt{c}n}{\zeta(\gamma)-\frac{1}{\gamma-1}}\zeta(\gamma,i)$:
    \begin{align*}
        \Pr&\left(D_{\ge i}(G_n)|>Cn\zeta(\gamma,i)\right)\le
        \Pr\left(D_{\ge i}(G_n)-\mathbb{E}(D_{\ge i}(G_n))>
            \frac{\sqrt{c}n}{\zeta(\gamma)-\frac{1}{\gamma-1}}\zeta(\gamma,i)\right)\le \\
           &\le\exp\left(-2\frac{cn}{{(\zeta(\gamma)-\frac{1}{\gamma-1})}^2}{\zeta(\gamma,i)}^2\right)
               \le\exp\left(-\frac{2cn}{{((\gamma-1)\zeta(\gamma)-1)}^2}\right)
    \end{align*}
    The probability that at least one of the bad events occur for $n$:
    \begin{align*}
        \Pr\big(\exists i\text{ s.t.\ }D_{\ge i}(G_n)> Cn\zeta(\gamma,i)\big)
               \le n\cdot \exp\left(-\frac{2cn}{{((\gamma-1)\zeta(\gamma)-1)}^2}\right).
    \end{align*}
    Let $\varepsilon:=\exp(-\frac{2c}{{((\gamma-1)\zeta(\gamma)-1)}^2})$ and
    $f(x):=\sum_{n=3}^\infty n\cdot{(\varepsilon x)}^n$. Since $\gamma>1$ and $c>0$,
    we have $\varepsilon<1$, therefore $f(1)$ is well-defined.
    The probability of not succeeding is
    \begin{align*}
        \Pr\Big(\exists n\ge 3\text{ s.t.\ }G_n\text{ does not satisfy
        \eqref{eq:powerlawdistrib}}\Big)&\le f(1).
    \end{align*}
    It is easy
    to see that $F(x)=\frac{(\varepsilon x)^3}{1-\varepsilon x}$ is a primitive
    function of $f(x)$ for $\varepsilon<1$.
    If $\varepsilon<\frac12$, then
    \begin{align*}
        f(1)\le\left(\frac{(\varepsilon x)^3}{1-\varepsilon
        x}\right)^\prime(1)=\frac{\varepsilon^3(3-2\varepsilon)}{(1-\varepsilon)^2}\le
        12\varepsilon^3.
    \end{align*}
    For $\gamma\ge 2$,
    \begin{equation*}
        \varepsilon\le\exp\left(-\frac{2c}{{(\zeta(2)-1)}^2}\right)<10^{-2c}.
    \end{equation*}
\end{proof}

Next we prove an improved lower bound on the matching number to be used later to
estimate the maximum degree.
\begin{lemma}[A generalized Vizing-bound]\label{lem:general_vizing}
    Let $G$ be a graph of order $n$, and let $d=d(G)$ be the degree sequence of
    $G$. Then
	\begin{equation*}
		\nu(G)\ge \max_{1\le q< n} \frac{e(G)-\sum_{i\ge q}i\cdot D_i(G)}{q}.
	\end{equation*}
\end{lemma}
\begin{proof}
Delete the vertices whose degree is at least $q$. By Vizing's theorem, there is a color class in the remaining graph whose size is at least $\frac{\chi'}{\Delta+1}$.
\end{proof}

\begin{lemma}
    For any $\gamma>2$, SF-DPG generates a distribution-bounded power-law
    degree sequence such that $\nu(G_n)\ge t(\gamma,c)\cdot n$ for all $n$,
    with high probability (as $c\to \infty$). The function $t(\gamma,c)$ is positive and depends
    on $\gamma$ and on the level of certainty $c$ from \Cref{lem:sfbound}.
\end{lemma}
\begin{proof}
    From \Cref{lem:sfbound} and then \Cref{lem:general_vizing}, we have, with high probability:
    \begin{align*}
        \nu(G_n)&\ge\max_{1\le q< n}
        \frac{e(G_n)-\sum_{i\ge q}i\cdot D_i(G_n)}{q}\ge\\
        & \ge\max_{1\le q< n}\frac{\frac12n-Cn\zeta(\gamma-1,q)}{q+1}\ge
        n\cdot\max_{1\le q< n}\frac{\frac12-\frac{C}{\gamma-2}{(q-1)}^{2-\gamma}}{q+1},
    \end{align*}
    where $C$ is defined on \cref{eq:sfboundC}.
Substituting $q={(\frac{4C}{\gamma-2})}^\frac{1}{\gamma-2}+1$, it follows that
    \begin{equation*}
        \nu(G_n)\ge  \frac{n}{4q+8}
    \end{equation*}
    holds for every $n\ge 2$ with high probability. The lower bound is indeed linear in $n$, since $C$
    only depends on the values of $c$~and~$\gamma$.
\end{proof}
The linearity of $\nu(G_n)$ implies that SF-DPG creates a vertex of degree
$\Omega(n^{1/(1-\gamma)})$ with high probability. In other words,
the degree sequences created by SF-DPG do not obey the more restrictive
power-law density-bound, where the maximum degree is $\mathcal{O}(n^{1/\gamma})$.

\paragraph{Further remarks and discussion.} Since every vertex in $G_{n-1}$ can
contribute with at most one edge to a matching, independently of their degree,
the formation of edges is not based on direct degree preference.  This is in
contrast with the Barab\'asi-Albert preferential attachment model, the
configuration model~\cite{Bollobas,Molloy}, and the Chung-Lu model~\cite{CL02}.

In general, it can be said that the process provides an ever-growing degree
sequence that is scale-free with the given parameter $\gamma>2$. However, it is
not clear how random the network $G_n$ is among all possible realizations of
its degree sequence. 

\medskip 

Fortunately, there is a known way to improve the quality of the sample. As Gao and
Greenhill proved in a recent paper (\cite{GG21}), such scale-free degree
sequences satisfy the so-called $P$-stable property. It ensures that the switch
Markov chain on the realizations of the generated degree sequence is mixing
rapidly. Therefore the application of the switch Markov chain will provide a
truly random realization of the degree sequence in polynomial time as a
function of the length of the degree sequence. (For details see~\cite{GG21}
or~\cite{P-stable}.) We expect that the advantage gained is that starting the switch Markov chain from the output of a DPG process should cut down on the relaxation time. Therefore, one can sample a random scale-free degree
sequence with the given parameter $\gamma > 2$, with a realization which is uniformly and
randomly chosen from all possible realizations.

\section{Regular graphs}\label{sec:reg}
Another useful version of the general DPG process is when the incoming vertex
degree is a constant $c$. If $c=2k$, then the process is rather straightforward. In
the case of an odd constant $c=2k+1$, only every second network will be $c$
regular, in the sequence. The number of edges in the $c$-regular network $G$ on $n$ vertices is
roughly $n c/2$. By Vizing's theorem, we have $\nu(G) \ge c/2$ so the
DP-step will succeed in each step.

\medskip

This dynamic model of ever growing $c$-regular network series  provides (relatively) uniformly random networks. Of course, it cannot be truly uniformly generated, since there are regular networks without a possible inverse DP-step. Therefore the network itself cannot be the result of a DPG process. (For $c=4$ see our \Cref{ex:K4}.) We will see below how we may ``randomize'' these networks.

\medskip

The DPG model for even-degree regular networks is not completely new. The case
$c=4$ played an important role for client-server architectures in peer-to-peer
networking, called SWAN technology~\cite{Bou03}. SWAN networks are very reliable and efficient TCP/IP
fabrics of connections~\cite{Holt05}. In this protocol, there is a 4-regular
dynamic network of TCP/IP network servers (or any other kinds of agents). A vertex
(agent) either wants to leave the network or wants to join to it. In the first case, an
DP-removal is performed on a degree 4 vertex (the agent leaves and the severed
links are reconnected). In the second case 2 edges are randomly chosen and a
forward DP-step is performed. In the original algorithm this step is called the
``clothespinning'' procedure. The same procedure was also used in 2D
vortex liquids analysis~\cite{HMM94}.

\medskip

Cooper, Dyer and Greenhill studied this clothespinning procedure in detail in
their rather technical paper~\cite{CDG07}. During the dynamic process the
network is decreasing and increasing in size, between some roughly predefined
boundaries. Their goal was to determine whether the clothespinning process
uniformly samples the set of all 4-regular networks between these size bounds.
Their complicated analysis gave an affirmative answer. According to their paper,
the same statement applies in general for any $2k$ regular networks as well.

\medskip

Intuitively that means that a regular graph generated by the DPG process can be
``randomized'' by a series of forward / inverse DP-steps. It remedies the
problem that there are regular networks that are unreachable by the ``regular''
forward DPG process (like \Cref{ex:K4}).

\section{The DP-removal of many vertices is NP-complete}\label{sec:NP}

As promised in \Cref{sec:prop}, we prove that deciding whether at least
$n^\varepsilon$ DP-removals ($\epsilon >0$) can be performed in a given graph is
\textsc{NP}-complete. The problem remains \textsc{NP}-hard even if the maximum
degree is a small constant. We will also see that the source of complexity is
not necessarily in finding the order in which the given set of vertices needs to be
removed.

\begin{observation}\label{obs:commutative}
    If $I\subset V(G)$ is an independent set, then the DP-removal of any two vertices of $I$ are commutative.
\end{observation}

\begin{observation}\label{obs:removable}
    If a vertex $x\in V(G)$ is DP removable from $G$ then
    $\overline{G}[\Gamma_G(x)]$ contains a matching of size $\lfloor
    \frac12d_G(x)\rfloor$.
\end{observation}

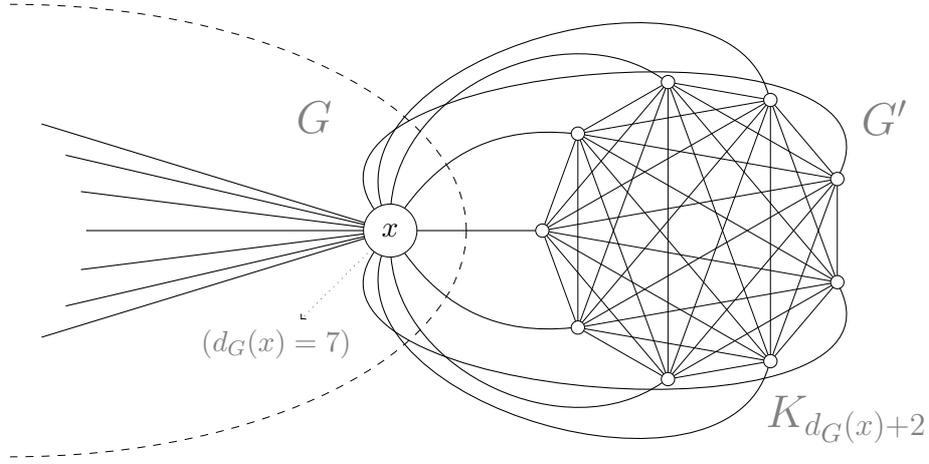
\begin{figure}[H]
    \centering
    \begin{tikzpicture}
    \begin{scope}
        \node[draw,circle,minimum size=20pt,inner sep=0pt,outer sep=0pt]
            (x) at (180:4) {$x$};
        \def\n{9}
        \pgfmathtruncatemacro{\nm}{\n-1}
        \pgfmathtruncatemacro{\nk}{\n-2}
        \pgfmathtruncatemacro{\nn}{\n-3}
        \foreach \i in {0,...,\nm}
        {
        \node[draw,circle,
        minimum size=5pt,inner sep=0pt,outer sep=0pt]
        (x\i) at
        ($ (\i*360/\n+2*90/\n:2) $) {};
        \draw[thin] (x) edge[bend right=(\i-\nm/2)*30] (x\i);
        }

        \foreach \i in {0,...,\nm}
        {
        \foreach \j in {1,...,\nm}
        {
            \pgfmathtruncatemacro{\jn}{mod(\i+\j,\n)}
            \ifthenelse{\jn>\i}
            {\draw[very thin] (x\jn)--(x\i);}{}
        }
        }

        \foreach \i in {0,...,\nn}
        {
        \draw[thin] (x) -- ($ (180:10)+(\i*90/\nn-45:2) $);
        }
        \draw[dashed,very thin] (-3,0) arc(0:90:6cm and 3cm);
        \draw[dashed,very thin] (-3,0) arc(0:-90:6cm and 3cm);
        \node[gray] at (-5,1.5) {\LARGE $G$};
        \node[gray] (v7) at (-5.5,-1.5) {($d_{G}(x)=\nk$)};
        \node[gray] at (2.5,1.5) {\LARGE $G'$};
        \node[gray] at (2,-2.5) {\LARGE $K_{d_{G}(x)+2}$};
        \draw[dotted,->] (x) -- (v7);
    \end{scope}
    \end{tikzpicture}
    \caption{Making a vertex $x\in V(G)$ non-DP-removable by joining 
    it to every vertex of a unique copy of a large enough
    clique.}\label{fig:NPhardPrevent}
\end{figure}
\begin{lemma}\label{lemma:NPhardPrevent}
    Suppose $x$ is a vertex of $G$. Let $K:=K_{d_{G}(x)+2}$ be a clique
    completely disjoint from $V(G)$. Let $G'$ be the disjoint union of $G$
    and $K$ plus ${d_{G}(x)+2}$ edges joining $x$ to every vertex of $K$,
    see \Cref{fig:NPhardPrevent}.  Then any sequence of DP-removals
    starting with $G'$ avoids removing $x$ and every vertex of $K$.
\end{lemma}
\begin{proof}
Proof by induction. Recall, that DP removal preserves edges induced between
non-removed vertices. By induction, the neighborhood of a $y\in V(K)$ is
$V(K)-y+x$, which induce a clique of size $d_G(x)$, so any sequence of
DP-removals starting with $G'$ avoids $y$. The vertex $x$ has $2d_G(x)+2$ neighbors: it contains the clique $K$ and $d_G(x)$ other vertices. Clearly, the maximum size matching in the complement of the neighborhood of $x$ is at most $d_G(x)$, which means that a pair of vertices in $V(K)$ are exposed (unmatched), so by     \Cref{obs:removable}, $x$ is not DP removable. The proof is now  complete.
\end{proof}

No matter the parity of $d_G(x)$, in the graph $G'$ constructed by
\Cref{lemma:NPhardPrevent}, the degree of $x$ becomes even. The degree of vertices
in the clique $V(K)$ in $G'$ is $d_G(x)+2$, which is odd if $d_G(x)$ is odd.  If one
wants to avoid adding odd degree vertices to $G$, then instead of joining $x$ to
nodes of a $K_{d(x)+2}$, one can take $K:=K_{\max(d(u),d(v))+2}$ for a
pair $u,v$ of odd degree vertices in $G$, and join both $u$ and $v$ to every vertex
of $K$.

\medskip

We are ready to prove the main result. The upper bound on the maximum degree in the following theorem is not optimized, at the moment we only care that it is a constant.

\begin{theorem}
Given a pair $(m,G)$, where $m$ is a positive integer and $G$ is a simple
graph, it is {\rm\textsc{NP}}-complete to decide whether it is possible to
DP-remove $m$ vertices from $G$. The problem remains {\rm\textsc{NP}}-hard even
when we restrict the input to $\Delta(G)\le 28$ and $m\le n^\varepsilon$
(where $n$ is the number of vertices of $G$ and $\varepsilon$ is a fixed
positive real).
\end{theorem}
\begin{proof}
The problem is trivially contained in \textsc{NP}, because checking the
DP-removability of one vertex is in P. We prove \textsc{NP}-hardness by a
\textbf{linear} reduction from \textsc{3-SAT-3} (max 3 literals in a clause,
every variable is present in max 3 clauses). We construct a graph $G$ with
$\Delta(G)\le 13$ such that a set of vertices can be completely removed via
DP-removals if and only if $\varphi$ is satisfiable, where $\varphi$ is in conjunctive normal form. By \Cref{lemma:NPhardPrevent}, this is sufficient to prove \textsc{NP}-hardness.

    \medskip

Let $X=\{x_1,\ldots,x_n\}$ be the set of variables of $\varphi$, and make them vertices of $G$. Let $H_i$ be a graph which contains $x_i$ and a copy of a $K_8-C_8$, see \Cref{fig:variable_gadget}.
    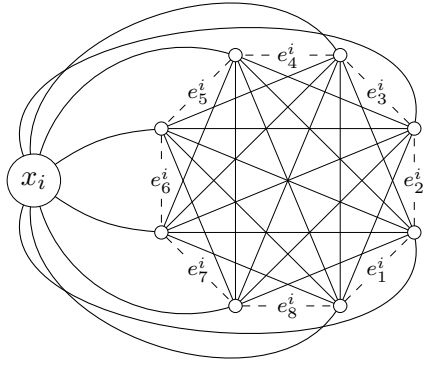
\begin{figure}[H]
        \centering
        \begin{tikzpicture}[scale=.9]
            \begin{scope}[rotate=90]
            \node[draw,circle,minimum size=20pt,inner sep=0pt,outer sep=0pt]
                (x0) at (0*120+90:6) {$x_i$};

            \foreach \k in {0}
            {
            \foreach \i in {0,...,7}
            {
            \node[draw,circle,minimum size=5pt,inner sep=0pt, outer
                sep=0pt]
            (x\k\i) at
            ($ (\k*120+90:2+0.5773502691896258/2) +
            (\i*360/8+13*360/16+\k*120:2) $) {};
            \draw[thin] (x\k) edge[bend right=(\i-3.5)*30] (x\k\i);
            }

            \foreach \i in {0,...,7}
            {
            \foreach \j in {2,...,6}
            {
                \pgfmathtruncatemacro{\jn}{mod(\i+\j,8)}
                \ifthenelse{\jn>\i}
                {\draw[very thin] (x\k\jn)--(x\k\i);}{}
            }
            }
            }

            \foreach \i in {0,...,7}
            {
            \pgfmathtruncatemacro{\jn}{mod(\i+1,8)}
            \pgfmathtruncatemacro{\ji}{mod(\i+2,8)+1}
            \pgfmathtruncatemacro{\jk}{mod(\i+3,8)+1}
            \pgfmathtruncatemacro{\jm}{mod(\i+6,8)+1}
            \draw[dashed] (x0\i)-- node[midway,circle,fill=white,inner
                sep=0pt,outer sep=0pt] {\scriptsize $e^i_\ji$} (x0\jn);
            }
        \end{scope}
        \end{tikzpicture}
        \caption{The variable gadget $H_i$. Dashed lines represent
        non-edges.}\label{fig:variable_gadget}
    \end{figure}
Let the dashed edges from \Cref{fig:variable_gadget} (forming a $C_8$) be $e^i_1,e^i_2,\dots,e^i_8$ in circular order; we will call these the literal edges (these edges are not incident on the vertex $x_i$). Let $C=\{c_1,\ldots,c_t\}$ be the set of clauses of $\varphi$. Let us find a function $f$ which maps
    \begin{align*}
        \{(x_i,c_\ell)\ :\ \neg x_i\in c_\ell\}&\longrightarrow
        \{e^i_{2j-1}\ :\ 1\le i\le n,\ 1\le j\le 4\},\\
        \{(x_i,c_\ell)\ :\ x_i\in c_\ell\}&\longrightarrow
        \{e^i_{2j}\ :\ 1\le i\le n,\ 1\le j\le 4\},
    \end{align*}
such that for any $1\le i\le n$ and any $1\le \ell<r\le t$ the edges $f(x_i,c_\ell)$ and $f(x_i,c_r)$ do not share any endpoints. Such an $f$ trivially exists, because each variable $x_i$ appears in at most 3 clauses of $\varphi$. For example, if $x_i\in c_k$ and $\neg x_i\in c_\ell,c_r$, then $f$ may map $(x_i,c_k)\mapsto     e^i_2$, $(x_i,c_\ell)\mapsto e^i_5$, $(x_i,c_r)\mapsto e^i_7$.

    \medskip

In $H_i$ there are exactly two ways to DP-remove the vertex $x_i$: if the DP-removal adds back the edges $\{e^i_{2j}\ |\ j=1,\ldots,4\}$ that shall represent that the value assigned to $x_i$ is \textbf{false}; if     the DP-removal adds back the edges $\{e^i_{2j-1}\ |\ j=1,\ldots,4\}$ that shall represent that the value assigned to $x_i$ is \textbf{true}.

    \medskip

Let us define $G$ formally first, which will be followed by an informal
description. We will use $\uplus$ to emphasize that the sets participating in
the union are disjoint. Essentially, we will assemble disjoint gadgets and
then identify specific parts of them to obtain the final structure.

Let $D=\{d_1,\ldots,d_t\}$ be a disjoint copy of the set $C$. Let $\sim$ be an equivalence relation that identifies a set of disjoint pairs of vertices of $\uplus_{i=1}^n V(H_i)$: for each $1\le \ell\le t$, arrange the two or three edges in $F(c_\ell):=\{ f(x_i,c_\ell)\ :\ \neg x_i\in c_\ell\text{ or }x_i\in c_\ell\}$ into a  triangle (there are exactly two ways to identify the vertices accordingly, choose arbitrarily) or a cherry (if $F(c_\ell)$ contains two edges). Record the pairs of overlapping     vertices into $\sim$ for every $c_\ell$. 
If $F(c_\ell)=2$, then we add an edge to $G$ between the dangling edges of the cherry $F(c_\ell)/\sim$.
Let
    \begin{align*}
        V(G)&:=\left(\uplus_{i=1}^n V(H_i)/\sim\right)\uplus C\uplus D,\\
        E(G)&:=\left(\uplus_{i=1}^n E(H_i)/\sim\right)\uplus \{c_\ell
        d_\ell\ :\ 1\le \ell\le t\}\uplus
        \{c_\ell v\ :\ v\in (\cup F(c_\ell))/\sim\} \uplus \\
        &\qquad \uplus\{ vw\ :\ v,w\in (\cup F(c_\ell))/\sim\text{ and }vw\notin F(c_\ell)/\sim \}.
    \end{align*}

See \Cref{fig:NPhard}. We continue with an informal description of $G$. Add $C$ to the set of vertices of $G$. If $c_\ell$ is a clause with three literals, we join $c_\ell$ to the endpoints of the at most three edges in $F(c_\ell)$. However, we \textbf{identify the endpoints of these edges} so that they form a triangle or a cherry. This increases the maximum degree of $G$ to $13$. Add a dummy vertex $d_\ell$ and join it to $c_\ell$ by an edge, thus making $d(c_\ell)=4$.

    \begin{figure}[H]
        \centering
        \begin{tikzpicture}[scale=1]
        \begin{scope}
            \node[draw,circle,minimum size=20pt,inner sep=0pt,outer sep=0pt]
                (x0) at (0*120+90:6) {$x_i$};
            \node[draw,circle,minimum size=20pt,inner sep=0pt,outer sep=0pt]
                (x1) at (1*120+90:6) {$x_k$};
            \node[draw,circle,minimum size=20pt,inner sep=0pt,outer sep=0pt]
                (x2) at (2*120+90:6) {$x_r$};

            \foreach \k in {0,1,2}
            {
            \foreach \i in {0,...,7}
            {
            \node[draw,circle,
            minimum size=5pt,inner sep=0pt,outer sep=0pt]
            (x\k\i) at
            ($ (\k*120+90:2+0.5773502691896258/2) + (\i*360/8+13*360/16+\k*120:2) $) {};
            \draw[thin] (x\k) edge[bend right=(\i-3.5)*7] (x\k\i);
            }

            \foreach \i in {0,...,7}
            {
            \foreach \j in {2,...,6}
            {
                \pgfmathtruncatemacro{\jn}{mod(\i+\j,8)}
                \ifthenelse{\jn>\i}
                {\draw[very thin] (x\k\jn)--(x\k\i);}{}
            }
            }
            }

            \node[draw,circle,minimum size=20pt,inner sep=0pt,outer sep=0pt]
                (cl) at (150:5) {$c_\ell$};
            \node[draw,circle,minimum size=20pt,inner sep=0pt,outer sep=0pt]
                (dl) at ($ (cl) + (-2,0) $) {$d_\ell$};

            \draw[very thick] (cl)--(dl);
            \draw[very thick] (cl) edge[bend left=-10] (x20);
            \draw[very thick] (cl) edge[bend left=10] (x00);
            \draw[very thick] (cl) -- (x10);

            \foreach \i in {0,...,7}
            {
            \pgfmathtruncatemacro{\jn}{mod(\i+1,8)}
            \pgfmathtruncatemacro{\ji}{mod(\i+2,8)+1}
            \pgfmathtruncatemacro{\jk}{mod(\i+3,8)+1}
            \pgfmathtruncatemacro{\jm}{mod(\i+6,8)+1}
            \draw[dashed] (x0\i)-- node[midway,circle,fill=white,inner
                sep=0pt,outer sep=0pt] {\scriptsize $e^i_\ji$} (x0\jn);
            \draw[dashed] (x1\i)-- node[midway,circle,fill=white,inner
                sep=0pt,outer sep=0pt] {\scriptsize $e^k_\jk$} (x1\jn);
            \draw[dashed] (x2\i)-- node[midway,circle,fill=white,inner
                sep=0pt,outer sep=0pt] {\scriptsize $e^r_\jm$} (x2\jn);
            }
        \end{scope}
        \end{tikzpicture}
        \caption{The clause gadget (a portion of the graph $G$)
        associated to $c_\ell=x_i\vee \neg x_k\vee
    x_r$.}\label{fig:NPhard}
    \end{figure}
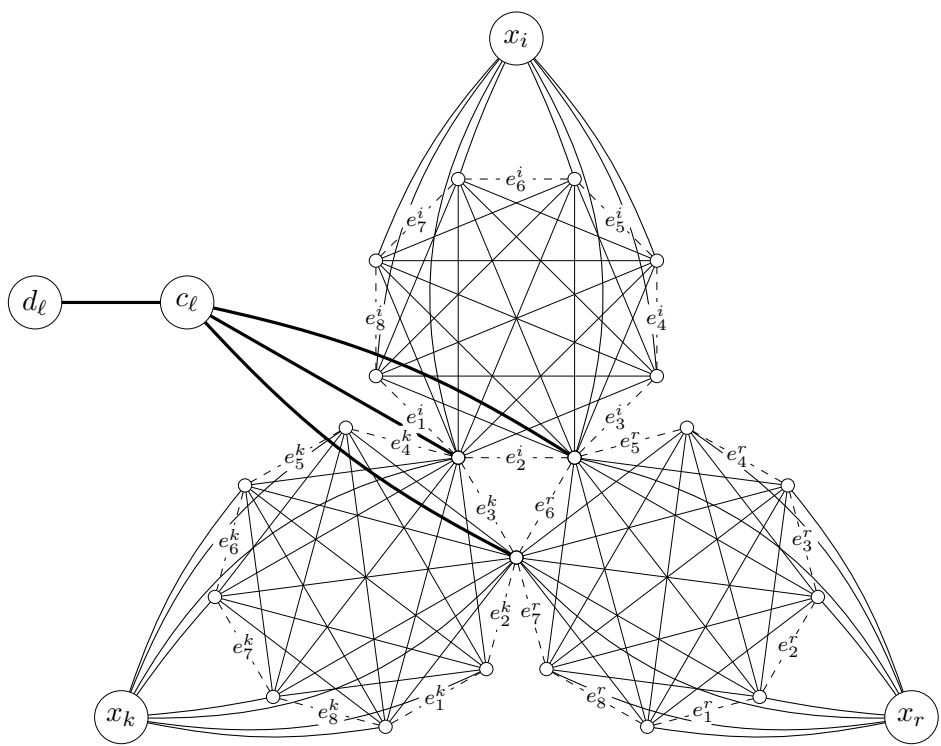

    \medskip

We claim that $X\cup C$ can be entirely DP-removed from $G$ if and only if $\varphi$ is satisfiable. Since $X\cup C$ is an independent set in $G$, the DP removal operations associated to its vertices are commutative (\Cref{obs:commutative}).  Suppose the variables $X$ are already DP-removed and let us study the clause $c_\ell$. The degree of $c_\ell$ in $G'$ is even, therefore only InvOp.~\ref{invop1} may apply to it. Since $d_\ell$ is isolated from the rest of the neighbors of $c_\ell$, there is a suitable matching from the non-edges of the neighborhood of $c_\ell$ if and only if at least one of the edges is missing among the endpoints of the literal edges (to which $c_\ell$ is joined); this happens precisely when at least one of the literals of $c_\ell$ is assigned     \textbf{true}.

    \medskip

To finish the proof, we apply \Cref{lemma:NPhardPrevent} to every vertex $v\in V(G)\setminus (X\cup C)$. In this way we obtain a $G'$ where the maximum degree is at most $13+15=28$. Since the extra cliques are completely disjoint from each other, the DP-removability of the rest of the vertices is not changed (see \Cref{obs:commutative}). By the linear size reduction from \textsc{3-SAT-3}, deciding whether $m=|X|+|C|$     vertices can be DP-removed from $G'$ is \textsc{NP}-complete.

\medskip

Restricting the input to $m\le n^\varepsilon$ is a mere technicality. Extend $G$ by the disjoint union of ${(|X|+|C|)}^\frac{1}{\varepsilon}$ copies of $K_{28}$; the extended graph is still a polynomial of the size of $\varphi$. Clearly, the existence of $|X|+|C|$ DP-removable vertices in $G$ and the extended graph are identical.

\end{proof}

\section{Open questions}

In summary, we presented a detailed analysis of the so-called degree-preserving growth dynamics of networks, which is a special case of degree saturation, a situation frequently occurring in real-world networks. We have shown that the general problem of deciding whether a graph can be built from a sequence of DPG
steps starting from a small kernel graph is NP-complete and finding the smallest such kernel graph is NP-hard. This, however, is in contrast with the numerical evidence from~\cite{DPG21} showing that most real-world networks can easily be constructed from DPG sequences, which raises the question as to what properties these networks have to share in order for this to be true? 
There are several other open questions that the DPG process raises, here are only a few: Can one (and how) characterize the irreducible graphs efficiently, that is, in a non-algorithmic fashion (currently we have to be checking every complement neighborhood for a maximum matching). Is there a degree sequence $d$ with many realizations such that every realization of it is DP-reducible? Given a degree sequence $d$, find a realization $G\in \mathbb{G}(d)$, which is irreducible. Given a multiset $D \in \mathbb{Z}^+$ and a graph $G_0$, under what conditions can one add $|D|$ vertices with degrees $D$, via DPG steps, starting from $G_0$?   Given two simple graphs, is there a sequence of forward and/or backward DP steps that can take one graph into the other? What is the variation distance from the uniform distribution when generating regular graphs via regular-DPG?

\section*{Appendix}
\textit{Irreducibility of \Cref{ex:K4}.}
    Let $V=\{v_1,\ldots,v_{4k}\}$ be a set of vertices ($k\ge 3$, subscripts
    are taken from $\mathbb{Z}/{4k}\mathbb{Z}$). Let the  4-regular graph $G_k$ be as follows: For all $i=1,\ldots,k$, let
    \begin{displaymath}
        G[v_{4i-3},v_{4i-2},v_{4i-1},v_{4i}]\simeq K_4,\quad\hbox{and}\quad
        v_{4i-1}v_{4i+1},v_{4i}v_{4i+2}\in E(G).
    \end{displaymath}
    Since $G$ is vertex-transitive, it is sufficient to show that $v_5$
    cannot be DP-removed.
    \begin{displaymath}
        \Gamma_G(v_5)=\{v_3,v_6,v_7,v_8\}\quad\hbox{and}\quad
        G[\Gamma_G(v_5)]\simeq K_3+\text{isolated vertex.} \qed
    \end{displaymath}

The following two statements fully describe the decomposable networks of max degree at most four.
\begin{lemma}\label{lemma:decomp3reg}
    If $\Delta(G)\le 3$, then $G$ can be DP-reduced into $K_3$'s, $K_4$'s, and maybe a triangle with a dangling stub and at most two other pendant edges. It is possible to achieve this without increasing the number of components of $G$. \qed
\end{lemma}
\begin{proof}
    If it is not possible to perform any of the inverse operations on a
    vertex $v$, then either $v$ is in a $K_4$ component of $G$, or $v$ is a
    first or second neighbor of the stub vertex $s$.

    \medskip

    A vertex $v$ of degree 2 can always be DP-removed except if $v$ is in a
    $K_3$ subgraph.
    A vertex $v(\neq s)$ of degree 1 can always be removed except if $v$ is
    the first or second neighbor of $s$.

    \medskip

    If the DP-removal of any vertex $v$ of degree 3 increases the number of
    components, then $G[\Gamma_G(v)]$ has zero edges, moreover, $G-v$ has exactly two
    more components than $G$. It is easy to see that there is a next inverse
    DP-step which decreases the number of components by one.

    \medskip

    If $vs$ is the stub edge and $v$ cannot be removed via InvOp.~\ref{invop3a},
    then $d(v)=3$ and the second neighbors of $s$ are joined by an edge.  If
    the graph is not decomposable, then every $x\in \Gamma_G(\Gamma_G(\Gamma_G(s)))$ must have
    degree 1.
\end{proof}

\begin{lemma}\label{lemma:4regindecomp}
    4-regular indecomposable graphs have the following structure: Take any number of vertex-disjoint copies of $K_5$, $K_5-e$, $K_4$. Then, to make the degree of every vertex equal to 4, match every vertex of degree 3 to
    another vertex of degree 3 in a different component, and add the edges corresponding to this matching to the graph. \qed
\end{lemma}
\begin{proof}[Sketch of the proof.]
    It is easy to confirm that if a graph is constructed as described in the
    claim of Lemma \ref{lemma:4regindecomp}, then it is irreducible, since there is an induced $K_3$
    in the neighborhood of every vertex.

    \medskip

    Let $G$ be a 4-regular indecomposable graph. For any $v\in V(G)$, one of the
    following must hold: \textbf{(A)} $G[\Gamma_G(v)]\simeq K_4$; \textbf{(B)}
    $G[\Gamma_G(v)]\simeq K_{4}-e$; \textbf{(C)}  $G[\Gamma_G(v)]\simeq K_3+\text{an
    isolated vertex.}$

    If $G[\Gamma_G(v)]$ does not contain a triangle, then $v$ can be DP-removed.
    If there are non-edges induced in $G[\Gamma_G(v)]$, then those must intersect
    at some vertex $w$, otherwise $v$ can be DP-removed.  However, if $w$
    only has one neighbor in $\Gamma_G(v)$, say $u$, then $\Gamma_G(w)=\{u,v,s,t\}$,
    where $s,t\notin \Gamma_G(v)$. Clearly, $w$ can be DP-removed, since all 4
    neighbors of $u$ and $v$ are already accounted for, and $s,t$ are not
    amongst them.

    \medskip

    Let us say that two arbitrary vertices $u$ and $v$ are in $\sim$
    relationship if $u$ is a vertex of a $K_3$ in $G[\Gamma_G(v)]$. It is easy to see
    that $\sim$ is symmetric. Also, if $u\sim v\sim w$, then $u,w\in \Gamma_G(v)$ and
    a short case analysis shows that $u\sim w$ in this case. Therefore, $\sim$
    is an equivalence relation. Each equivalence class induces exactly one of
    the three components listed in the statement of the lemma, because we
    already know what the neighborhoods look like. Since the edges joining
    vertices from two distinct equivalence classes must join degree 3 vertices,
    these form a matching.
\end{proof}

\end{document}